\documentclass[11pt]{article}
\usepackage[latin1]{inputenc}
\usepackage[margin=0.5in]{geometry}
\usepackage{amsmath,amssymb,amstext}
\usepackage{cleveref}
\usepackage{amsthm}
\usepackage{authblk}
\usepackage{enumitem}
\usepackage{dirtytalk}
\usepackage{tikz}
\usepackage{thmtools}
\usepackage{thm-restate}
\usetikzlibrary{shapes}

\declaretheorem[name=Theorem,numberwithin=section]{thm}
\newtheorem{defs}[thm]{Definition}
\newtheorem{claim}[thm]{Claim}
\newtheorem{lemma}[thm]{Lemma}

\newtheorem{prop}[thm]{Proposition}
\newtheorem{obs}[thm]{Observation}

\newtheorem{conjecture}[thm]{Conjecture}
\newtheorem{question}[thm]{Question}

\begin{document}

\author{Luke Postle}
\author{Evelyne Smith-Roberge$^*$}
\affil{Dept. of Combinatorics and Optimization, University of Waterloo, 
Canada. \\ \texttt{\{lpostle, e2smithr\}@uwaterloo.ca}}

\title{On the Density of $C_7$-Critical Graphs}
\date{March 2019}
\maketitle
\abstract{In 1959, Gr\"{o}tzsch \cite{grotzsch1959dreifarbensatz} famously proved that every planar graph of girth at least 4 is 3-colourable (or equivalently, admits a homomorphism to $C_3$).  A natural generalization of this is the following conjecture: for every positive integer $t$, every planar graph of girth at least $4t$ admits a homomorphism to $C_{2t+1}$. This is in fact the planar dual of a well-known conjecture of Jaeger \cite{jaeger1984circular} which states that every $4t$-edge-connected graph admits a modulo $(2t+1)$-orientation.  Though Jaeger's original conjecture was disproved in \cite{han2018counterexamples}, Lovasz et al. \cite{lovasz2013nowhere} showed that every $6t$-edge connected graph admits a modulo $(2t+1)$-flow. The latter result implies that every planar graph of girth at least $6t$ admits a homomorphism to $C_{2t+1}$. We improve upon this in the $t=3$ case, by showing that every planar graph of girth at least $16$ admits a homomorphism to $C_7$.  We obtain this through a more general result regarding the density of $C_7$-critical graphs: if $G$ is a $C_7$-critical graph with $G \not \in \{C_3, C_5\}$, then $e(G) \geq \tfrac{17v(G)-2}{15}$.}  
\begin{section}{Introduction and Notation}

In 1951, Dirac \cite{dirac1951note} introduced the concept of colour-criticality and since then, colour-critical graphs have been widely studied. A graph $G$ is $k$-\emph{critical} if its chromatic number is $k$ and the chromatic number of every proper subgraph of $G$ is strictly less than $k$. As every graph with chromatic number $k$ contains a $k$-critical subgraph, it is useful to study $k$-colourability via colour-critical graphs. More generally, it is useful to study graph homomorphisms\footnote{A homomorphism $\phi:G \rightarrow H$ from a graph $G$ to a target graph $H$ is a mapping of the vertices of $G$ to those of $H$, such that for each edge $uv \in E(G)$, $\phi(u)\phi(v) \in E(H)$.} through homomorphism-critical graphs, which we define as follows.
\begin{defs}
Let $H$ be a graph. A graph $G$ is \textnormal{$H$-critical} if every proper subgraph of $G$ admits a homomorphism to $H$, but $G$ itself does not. 
\end{defs}

Perhaps one of the more famous results concerning homomorphisms of planar graphs is Gr\"{o}tzsch's Theorem \cite{grotzsch1959dreifarbensatz}, which states that every planar graph of girth at least $4$ admits a homomorphism to $C_3$ (or equivalently, is 3-colourable).  As a natural generalization of this, one might conjecture the following.

\begin{conjecture}\label{jaeger}
If $G$ is a planar graph of girth at least $4t$, then $G$ admits a homomorphism to $C_{2t+1}$. 
\end{conjecture}
 
This is in fact the planar dual of a well-known conjecture of Jaeger \cite{jaeger1984circular}  which states that every $4t$-edge-connected graph admits a modulo $(2t+1)$-orientation\footnote{That is, an orientation of its edges such that for each vertex, the difference of the in-degree and the out-degree is congruent to 0 modulo $2t+1$.}. Though Jaeger's original conjecture was shown to be false in early 2018 \cite{han2018counterexamples}, all counterexamples found thus far are non-planar. As such, Conjecture \ref{jaeger} is still open. We note that Conjecture \ref{jaeger} is equivalent to the following:  if $G$ is a planar graph of girth at least $4t$, then $G$ admits a $\left(\tfrac{2t+1}{t} \right)$-circular colouring. For an overview on circular colouring, see \cite{surv}. The $t=1$ case is the only case in which the conjecture has been confirmed. 
 
Considerable progress has been made in the general $t$ case, though the girth bound of $4t$ remains elusive. In 1996, Ne\v set\v ril and Zhu \cite{NesAndZhu} showed that every planar graph of girth at least $10t-4$ admits a homomorphism to $C_{2t+1}$. In 2000, Klostermeyer and Zhang \cite{klostermeyer20002+} showed that it is sufficient to bound the odd girth\footnote{The \emph{odd girth} of a graph is the length of its shortest odd cycle.} of the graph as being at least $10t-3$. A year later, Zhu \cite{zhu2001circular} showed that a girth of at least $8t-3$ is sufficient, and in 2003, Borodin et al. \cite{borodinEtAl} improved upon this by showing a girth of at least $\frac{20t-3}{3}$ suffices. Progress stalled for a decade until in 2013, Lov\'asz et al. \cite{lovasz2013nowhere} showed that $6t$-edge-connected graph admits a modulo $(2t+1)$-orientation. As a corollary to this, they obtain that every planar graph of girth at least $6t$ admits a homomorphism to $C_{2t+1}$.  This is the best known general bound, though in the $t=2$ case Dvo{\v{r}}{\'a}k and Postle  \cite{dvovrak2017density} showed that every planar graph of odd girth at least 11 (and hence of girth at least 10) admits a homomorphism to $C_5$. 

Our first main result is the following.  

\begin{thm}\label{girth16}
If $G$ is a planar graph with girth at least 16, then $G$ admits a homomorphism to $C_7$.
\end{thm}

This stems from Theorem \ref{main}, below, in which we bound the density of $C_7$-critical graphs. We note Cranston and Li proved Theorem \ref{girth16} independently in \cite{cranston2020circular}, as a consequence of a more technical theorem regarding graph boundaries. (Though the consequences of our works for Conjecture \ref{jaeger} are the same, our density theorem and their theorem do not imply each other.) A trivial density bound for $C_{2t+1}$-critical graphs arises from the fact that they have minimum degree 2 (see Lemma \ref{cut}). This tells us that if $G$ is a $C_{2t+1}$-critical graph, then $e(G) \geq v(G)$. Unfortunately, we cannot beat this bound in the general case as for $t \geq 1$, the $(2t-1)$-cycle is $C_{2t+1}$-critical. However, we can do better if we assume $G$ contains a vertex of degree at least 3. In this case, a straightforward discharging argument shows that if $G$ is a $C_{2t+1}$-critical graph that contains a vertex of degree at least 3, then $e(G) \geq \left( 1 + \tfrac{1}{4t}\right)v(G) + \tfrac{1}{3t}$ (see Lemma \ref{basic-density}).

Ours are not the first density results regarding $C_{2t+1}$-critical graphs. In \cite{borodinEtAl}, Borodin et al. show that if $G$ is a $C_{2t+1}$-critical graph with girth at least $6t-2$, then $G$ contains a subgraph $G'$ with $e(G') \geq (1+\tfrac{3}{10t-4})v(G')$. In \cite{dvovrak2017density}, Dvo{\v{r}}{\'a}k and Postle give the best-known result for the $t=2$ case by showing that every $C_5$-critical graph on at least four vertices has $e(G) \geq \frac{5v(G)-2}{4}$.  

Our second main result, which concerns the density of $C_7$-critical graphs, is the following.

\begin{restatable}{thm}{main}
\label{main} Let $G$ be a $C_7$-critical graph. If $G \not \in \{C_3, C_5\}$, then $e(G) \geq \tfrac{17v(G)-2}{15}$.
\end{restatable}

From Theorem \ref{main} and using Euler's formula for graphs embedded in surfaces, we immediately obtain the following result. 
 
\begin{thm}\label{girth17}
If $G$ is a planar or projective planar graph of girth at least 17, then $G$ admits a homomorphism to $C_7$.
\end{thm}

In order to further lower the girth bound to 16 in the planar case and obtain Theorem \ref{girth16}, we use the following lemma of Klostermeyer and Zhang \cite{klostermeyer20002+}.
\vskip 2mm
\noindent{\textbf{Folding Lemma (Klostermeyer and Zhang \cite{klostermeyer20002+})}  \emph{Let $G$ be a planar graph with odd girth $k$. If $C = v_0 \dots v_{r-1}v_0$ is a cycle in $G$ that bounds a face and $r \neq k$, then there is an integer $i \in  \{0,...,r-1\}$ such that the graph $G'$ obtained from $G$ by identifying $v_{i-1}$ and $v_{i+1}$ (mod $r$) is of odd girth $k$.}}
\vskip 2mm
With this, we obtain from Theorem \ref{girth17} the following theorem:
\begin{thm}\label{oddcycles}
If $G$ is a planar graph with odd girth at least 17, then $G$ admits a homomorphism to $C_7$.
\end{thm}

\begin{proof}
By the Folding Lemma, we may assume a minimum counterexample to Theorem \ref{oddcycles} only has faces of length 17. The theorem now follows directly from Theorem \ref{main} and Euler's formula for planar graphs. 
\end{proof}

\subsection{Outline}
In Section 2, we will present relevant definitions, give general results regarding $C_{2t+1}$-critical graphs, and provide an outline of the proof of Theorem \ref{main}.   In Section 3, we will establish the structural properties of a minimum counterexample to Theorem \ref{main}. Section 4 is dedicated to proving Theorem \ref{main} using discharging and the structure established in the foregoing section. The discharging process used is rather intricate: for one thing, the discharging rules are performed sequentially. This ensures that in the later steps of the discharging process, much of the local structure surrounding the structures receiving charge is known. Additionally, charge is only ever sent along short, nearby paths with internal vertices of degree two. If a structure sends charge to many vertices, it follows from the initial charges that the structure sending charge has large amount of charge to spare. We believe a similar form of sequential discharging may prove useful in proving Conjecture \ref{jaeger}. In Section 5, we will discuss the limits of the methods used here in establishing possible improvements to Theorem \ref{girth16} and Conjecture \ref{jaeger}. 
\end{section}
\begin{section}{Preliminaries}

In this section, we present several results concerning $C_{2t+1}$-critical graphs. In Subsection \ref{Strings}, we establish local structural results regarding the surroundings of vertices of degree at least three. The subsection ends with the proof of Lemma \ref{basic-density}. In Subsection \ref{Cells}, we bound the number of vertices of degree two surrounding $(2t+1)$-cycles in $C_{2t+1}$-critical graphs. In Subsection \ref{Potential}, we define all relevant terms to the \emph{potential method} used in the proof of Theorem \ref{main}. Finally, the section ends with an overview of the structure of the proof of Theorem \ref{main}.

We note that in general, it is only interesting to study $H$-critical graphs when $H$ is a \emph{core}: a graph that does not admit a homomorphism to a proper subgraph of itself. Moreover, by limiting our study to graphs that are $H$-critical for a vertex-transitive core $H$, we ensure the graphs under study are 2-connected. Indeed, we have the following. 

\begin{lemma}\label{cut}
\textbf{\emph{(Folklore)}} Let $H$ be a vertex-transitive graph, and let $G$ be an $H$-critical graph. Then $G$ is 2-connected. 
\end{lemma}

This follows from the fact that if $H$ is vertex-transitive and $G$ is an $H$-critical graph containing a cut vertex $v$, partial homomorphisms of $G$ to $H$ may be combined at $v$ contradicting the fact that $G$ is $H$-critical.

For a set $S \subseteq V(G)$, the neighbourhood of $S$ is denoted $N(S)$ and is defined as $N(S) = \cup_{v \in S} N(v)$. The path with $t$ edges is denoted $P_t$, and will be referred to as the path of \emph{length} $t$. The \emph{internal vertices} of a path are the vertices of the path that are not endpoints of the path.

\subsection{Strings}\label{Strings}

A crucial part of our analysis of graph homomorphisms consists of examining the extensions of partial homomorphisms to the entire graph. Paths with internal vertices of degree 2 play an important role in our investigation, as it is easy to determine the extensions of a partial homomorphism along such paths. In addition, the low-density $C_7$-critical graphs under study in further sections contain a relatively high amount of vertices of degree two. As a consequence, such paths are ubiquitous. For these reasons, we define the following terms.

\begin{defs}\label{def:string}
A \emph{string} in a graph $G$ is a path with internal vertices of degree two and endpoints of degree at least three. A \emph{$k$-string} is a string with $k$ internal vertices. We say a vertex is \emph{incident with a string} if it is an endpoint of the string. Two vertices \emph{share} a string if they are the endpoints of that string.
\end{defs}

If a vertex is incident with many long strings, then its local density is relatively low.  As we aim to lower-bound the density of $C_{2t+1}$-critical graphs, it is useful to be able to bound the number of degree two vertices in the strings incident with vertices of degree at least three.

First, we define the following.
\begin{defs}
Let $G$ be an $H$-critical graph for some graph $H$.  Let $u$ and $v$ be vertices on a path $P$ in $G$ such that the internal vertices of $P$ have degree 2 in $G$. Let $\phi: u \rightarrow H$ be a homomorphism. Let $\Phi$ be the set of extensions of $\phi$ to $P$. We define $B_\phi(v | u, P ) := \{\phi'(v) : \phi' \in \Phi\}$. If the choice of $\phi$ is irrelevant (for instance if we only wish to speak of $|B_\phi(v|u,P)|$), we will sometimes write $B(v|u, P)$.
\end{defs}

We will use the following observation.
\begin{obs}\label{nbrs} If $H$ is an odd cycle, then for any nonempty $S \subsetneq V(H)$, $|N(S)| > |S|$. 
\end{obs}

 We note that the above observation as well as Lemmas \ref{vartheta}, \ref{max-str}, and \ref{weight} readily generalize to $H$-critical graphs when $H$ is any non-bipartite, vertex-transitive graph. 
\begin{lemma}\label{vartheta}
Let $G$ and $H$ be graphs, and suppose $H$ is an odd cycle. Let $P = v_0v_1...v_{k+1}$ be a path in $G$ with $k+1$ edges, with $\deg_G(v) = 2$ for each  $v \in V(P)\setminus\{v_0, v_{k+1}\}$. Let $\phi: v_0 \rightarrow H$ be a homomorphism. Then $|B_\phi(v_{k+1} | v_0, P)| \geq  \min(k+2, v(H))$.
\end{lemma}

This is a consequence of Observation \ref{nbrs} and is proved by Borodin et al. in Example 2.2 of \cite{borodinEtAl}.

We are now equipped to restrict the length of strings in $C_{2t+1}$-critical graphs.
\begin{lemma}\label{max-str}
Let $H$ be an odd cycle. If $G$ is an $H$-critical graph, then $G$ does not contain a $k$-string with $k \geq v(H)-2$. 
\end{lemma}

\begin{proof}
Suppose not: that is, suppose $P = v_0v_1v_2\cdots v_{v(H)-1}$ is a subpath of a string in $G$. Since $G$ is $H$-critical, $G - \{v_1, \cdots v_{v(H)-2}\}$ admits a homomorphism $\phi$ to $H$. By Lemma \ref{vartheta},  $|B_\phi(v_{v(H)-1} | v_0, P)| \geq \min(v(H), v(H)) = v(H)$. Hence $\phi(v_{v(H)-1}) \in B_\phi(v_{v(H)-1}|v_0, P)$, and so $\phi$ extends to $G$. This contradicts the fact that $G$ is $H$-critical.
\end{proof}

Next, we show that $C_{2t+1}$-critical graphs do not contain two vertices that share distinct strings with the same number of vertices modulo 2. 

\begin{lemma}\label{par-str}
Let $H$ be an odd cycle, and let $G$ be $H$-critical. Let $S = uv_1v_2\cdots v_k v$ be a $k$-string in $G$, and let $P = uw_1w_2 \cdots w_tv$ be a $(u,v)$-path in $G-E(S)$. Then either $t > k$, or $t \not \equiv k \mod 2$.
\end{lemma}
\begin{proof}
Suppose not. As $G$ is $H$-critical, $G-E(S)$ admits a homomorphism $\phi$ to $H$. But $\phi$ extends to $G$ in the following way: for $i = 1,\cdots, t$, let $\phi(v_i) = \phi(w_i)$. For $t < i \leq k$ with $i \not \equiv t \mod 2$, let $\phi(v_i) = \phi(v)$. For $t < i < k$ with $i \equiv t \mod 2$, let $\phi(v_i) = \phi(w_t)$. This contradicts the fact that $G$ is $H$-critical.
\end{proof}

Lemma \ref{max-str} gives us a bound on the local density surrounding a vertex of degree at least 3, but a better bound arises by considering the entire set of strings incident with the vertex rather than each string individually. To that end, we define the \emph{weight} of a vertex as follows.

\begin{defs}
Let $G$ be a graph, and let $v \in V(G)$ be a vertex of degree $d \geq 3$, and let $k_1, k_2, \dots, k_d$ be integers with $k_1 \geq \dots \geq k_d$.  If $v$ is incident with $d$ distinct strings $S_1, \dots, S_d$ where $S_i$ is a $k_i$-string for each $1 \leq i \leq d$, we say $v$ is of \emph{type} $(k_1, \dots, k_d)$.  If $v$ is a vertex of type $(k_1, ..., k_d)$, we define the \emph{weight of $v$} as $\textrm{wt}(v) = \sum_{i=1}^d k_i$.
\end{defs}

Recall that by Lemma \ref{cut}, if $G$ is a $C_{2t+1}$-critical graph then $G$ does not contain a vertex that is both endpoints of a string: hence the type of a vertex in a $C_{2t+1}$-critical graph is well-defined. 

We bound the weight of vertices as follows. Note the lemma below and its proof are found in \cite{zhu2001circular} (Lemma 3.3).

\begin{lemma}\label{weight} If $H$ is an odd cycle, $G$ is an $H$-critical graph, and $v \in V(G)$, then
$\emph{\textrm{wt}}(v) \leq (v(H)-2)\deg(v)-v(H). 
$
\end{lemma}


Having established the required definitions, we are now equipped to show the following.

\begin{lemma}\label{basic-density}
If $G$ is a $C_{2t+1}$-critical graph that contains a vertex of degree at least 3, then \begin{equation*}
    e(G) \geq (1 + \tfrac{1}{4t})v(G) + \tfrac{1}{3t}.
\end{equation*} 
\end{lemma}

\begin{proof}
Suppose not. Let $G$ be a counterexample. Since $e(G) < (1+ \tfrac{1}{4t})v(G) + \frac{1}{3t}$, it follows that $4te(G)-(4t+1)v(G) < \tfrac{4}{3}$. We will assign an initial charge of $ch_0(v) = 4t\deg(v)-8t-2$ to each vertex $v \in V(G)$, so that $\sum_{v \in V(G)} ch_0(v) = \sum_{v \in V(G)} (4t\deg(v)-8t-2) = 2(4te(G)-(4t+1)v(G)) < \tfrac{8}{3}$.

We discharge according to the following rule to obtain a final charge $ch_1(v)$ for each vertex $v \in V(G)$: each vertex of degree $2$ sends $-1$ to the endpoints of the string that contains it. We claim that after discharging every vertex of degree at least 3 has positive charge. 

To see this, let $v$ be a vertex of degree at least 3. Then $ch_1(v) = 4t\deg(v) -8t-2 - \textnormal{wt}(v)$. By Lemma \ref{weight}, $\textnormal{wt}(v) \leq (2t-1)\deg(v) - (2t+1)$, and so it follows that $ch_1(v) \geq  (2t+1)\deg(v)-6t-1$. Since $\deg(v) \geq 3$ and $t \geq 1$, $ch_1(v) \geq 2$. 

Since $v$ is a vertex of degree at least 3, it is the endpoint of a string $S$. By Lemma \ref{cut}, $S$ has two distinct endpoints, each of which have degree at least $3$ by the definition of string. Thus $G$ contains at least two vertices of degree 3. Since every vertex of degree 2 has final charge 0 and $G$ contains at least two vertices of degree 3, it follows that the total sum of the charges is at least 4 \textemdash a contradiction.
\end{proof}

\subsection{Cells}\label{Cells}
Cycles of length seven will play an important role in establishing the structure of $C_7$-critical graphs in the following sections. For that reason, we define the following terms.

\begin{defs}\label{celldefs}
A $(2t+1)$-cycle in a $C_{2t+1}$-critical graph is called a \emph{cell}\footnote{We caution the reader that we do not mean to suggest that there is an intrinsic property of $(2t+1)$-cycles that makes them a vital part of the structural analysis for general $C_{2t+1}$-critical graphs. For $C_7$-critical graphs, cells as defined proved a useful tool in our analysis. For values of $t$ larger than 3, it seems likely that $(2t+3)$-cycles and perhaps $(2t+5)$-cycles will prove equally useful in establishing the structure of $C_{2t+1}$-critical graphs.}. A cell $C$ is \emph{incident} with a string $S \not \subseteq C$ if one of the endpoints of the string is contained in the cell.
\end{defs}

In the discharging process in Section 4, cells aggregate and dispense charge in the graph in much the way vertices do. In this way, we think of cells as elementary structures, and treat them as supervertices.  It is therefore unsurprising that notions of cell \emph{weight} and \emph{type} (analogous to their vertex counterparts) prove useful in our analysis.

\begin{defs} The \emph{degree of a cell $C$} is the number of strings incident with $C$. Let $\deg(C) = d$, and let $k_1, k_2, ..., k_d$ be integers with $k_1 \geq ... \geq k_d$. If $C$ is incident with $d$ distinct strings $S_1, \dots, S_d$ where for $1 \leq i \leq d$ $S_i$ is a $k_i$-string, we say $C$ is a cell of \emph{type} $(k_1, \dots, k_d)$. If $C$ is a cell of type $(k_1, ..., k_d)$, we define the \emph{weight of $C$} as $\textrm{wt}(C) = \sum_{i=1}^d k_i$.
\end{defs}

Note the definition does not preclude a cell $C$ from containing both endpoints of a string $S$ with $S \not \subset C$. However, we show in the following lemma that this does not happen. 

\begin{lemma}\label{stringsincells}
Let $t \geq 1$ be an integer, and let $C$ be a cell in a $C_{2t+1}$-critical graph $G$. Let $S \not \subseteq C$ be a string. At most one of the endpoints of $S$ is contained in $V(C)$.
\end{lemma}
\begin{proof}
Note that if $G$ contains a string, it contains two vertices of degree at least three and so is not an odd cycle. Since $G$ is $C_{2t+1}$ critical, $G$ does not contain an odd cycle with fewer than $2t+1$ edges, as such a cycle is $C_{2t+1}$-critical itself. The proof then follows easily from Lemma \ref{par-str}. 
\end{proof}

In the spirit of Lemma \ref{weight}, the following lemma provides some restriction on the local structure surrounding a cell in a $C_{2t+1}$-critical graph.

\begin{lemma}\label{cell-weight}
Let $G$ be a $C_{2t+1}$-critical graph. If $C$ is a cell of $G$, then
$\emph{\textrm{wt}}(C) \leq (2t-1)\deg(C)-(2t+1)$.
\end{lemma}
\begin{proof}
Let $C$ be a cell of $G$ with $\deg(C) = d$. By Lemma \ref{stringsincells}, $C$ is a cell of type $(k_1, ..., k_d)$, incident with a $k_i$-string $S_i$ for each $1 \leq i \leq d$. We will denote by $c_i$ and $v_i$ the endpoints of each $S_i$, with $v_i \not \in V(C)$. 

Note first there are $2(2t+1)$ homomorphisms of a cell to $C_{2t+1}$.  Given a homomorphism $\phi: v_i \rightarrow C_{2t+1}$, we denote by $B_\phi(C|v_i, S_i)$ the set of possible extensions of $\phi$ to $S_i \cup C$. Note that $ |B_\phi(C|v_i, S_i)| = 2 |B_\phi(c_i| v_i, S_i)|$. Since $G$ does not admit a homomorphism to $C_{2t+1}$, we have $\cap_{i=1}^d B(C |v_i, S_i) = \emptyset$. Therefore $\sum_{i=1}^d (2(2t+1)-|B(C|v_i, S_i)|) \geq 2(2t+1)$. By Lemma \ref{vartheta}, for each $1 \leq i \leq d$ we have $|B(c_i | v_i, P_i)| \geq \min(k_i+2, 2t+1)$. From Lemma \ref{max-str}, since $G$ is $C_{2t+1}$-critical, $k_i+2 < 2t+1$. Therefore $|B(c_i|v_i, P_i)| \geq k_i+2$, and since $|B(C|v_i, S_i)|)| = 2|B(c_i|v_i, P_i)|$ for each $i$, it follows that $\sum_{i=1}^d (2(2t+1)-2(k_i+2)) \geq 2(2t+1)$. Using the fact that $\sum_{i=1}^d k_i = \textrm{wt}(C)$, dividing by 2, and reorganizing, we obtain $\textrm{wt}(C) \leq (2t-1)\deg(C)-(2t+1)$, as desired.
\end{proof}

\subsection{Potential}\label{Potential}

Our analysis will also rely heavily on the insights gained from identifying vertices in a critical graph and examining the resulting graph. It is useful to be able to speak of undoing the identification process; to that end, we define the following term.

\begin{defs}
Let $u$ and $v$ be non-adjacent vertices in a graph $G$. Let $G'$ be the graph obtained from $G$ by identifying $u$ and $v$ to a new vertex $z$. Let $H$ be a subgraph of $G'$ that contains $z$. Given the identification of $u$ and $v$ to $z$, \emph{splitting $z$ back into $u$ and $v$} refers to deleting $z$ and adding new vertices $u$ and $v$ to $V(H)$, and for each $x \in \{u,v\}$, adding to $E(H)$ all edges of the form $xy$ such that $y \in V(H)$ and $xy \in E(G)$. 
\end{defs}

In both Sections 3 and 4, we will use \emph{potential} to learn about the density of subgraphs of minimum counterexample to Theorem \ref{main}. The potential method used here was popularized by Kostochka and Yancey in \cite{kostochka2014ore} in order to give a lower bound on the number of edges in colour-critical graphs. In \cite{dvovrak2017density}, Dvo{\v{r}}{\'a}k and Postle use potential to bound the density of $C_5$-critical graphs. 

We define potential as follows.

\begin{defs}
Let $\alpha, \beta >0 $ and let $G$ be a graph. The $(\alpha, \beta)$-\emph{potential} of $G$ is given by
\begin{equation*}
    p_{\alpha, \beta}(G) = \alpha v(G) - \beta e(G). 
\end{equation*}
\end{defs}

When $\alpha$ and $\beta$ are clear from the context, we will omit them and speak only of the \emph{potential} of a graph and its subgraphs. Potential on its own is merely a measure of the density of the graph: what makes it a powerful tool for structural analysis is the reduction found in the paragraphs below.

In addition, we will require the following definition. 
\begin{defs}\label{phi(F)}
Let $\phi: G \rightarrow H$ be a homomorphism, and let $F \subseteq G$. We let $\phi(F)$ denote the subgraph of $H$ with $V(\phi(F)) = \{\phi(v) : v \in V(F)\}$ and $E(\phi(F)) = \{\phi(u)\phi(v) : uv \in E(F)\}$.
\end{defs}

\begin{defs}\label{gsubphi}
Let $G$ be an $H$-critical graph, and let $F$ be a proper subgraph of $G$. Let $\phi:F \rightarrow H$ be a homomorphism. Let $G'$ be the graph with $V(G') = V(G\setminus F) \cup V(\phi(F))$, and $E(G') = E(G\setminus F) \cup E(\phi(F))$. For each $u \in \phi(F)$, let $\phi^{-1}(u)$ be the set of vertices of $F$ with image $u$ under $\phi$.   We define $G_F[\phi]$ as the graph obtained from $G'$ by adding an edge $vu$ for each $u \in \phi(F)$ and $v \in V(G)\setminus V(F)$ such that there exists $w \in \phi^{-1}(u)$ with $vw \in E(G)$.
\end{defs}

Let $G, H, F,$ and $\phi$ be as in Definition \ref{gsubphi}. Note $G_F[\phi]$ has no homomorphism to $H$, as such a homomorphism $\phi'$ admits an extension to a homomorphism  $\phi'': V(G) \rightarrow V(H)$ by setting $\phi''(v) = \phi'(\phi(v))$ for each $v \in V(F)$, and $\phi''(v) = \phi'(v)$ for each $v \in V(G \setminus F)$. Thus $G_F[\phi]$ contains an $H$-critical subgraph $W$. Note if $F$ is not isomorphic to a subgraph of $H$, then $G_F[\phi]$ contains fewer vertices than $F$, and hence $W$ contains fewer vertices than $G$. Furthermore, $W \cap \phi(F) \neq \emptyset$ as otherwise $W \subset G$ and so $G$ contains a proper $H$-critical subgraph, contradicting the fact that $G$ is $H$-critical.  This motivates the following definitions.

\begin{defs}\label{potdefs}
Let $G$ be a graph. The subgraph $F'$ of $G$ is an \emph{extension} of $F \subsetneq G$ if there exists a homomorphism $\phi:F \rightarrow H$ and an $H$-critical subgraph $W$ in $G_F[\phi]$ such that $V(F') = V(W \setminus \phi(F)) \cup V(F)$, and $E(F') = E(W\setminus \phi(F)) \cup E(F) \cup E$, where $E$ is a minimal set of edges containing an edge $vu$ for each $u \in \phi(F)$ and $v \in V(W \setminus \phi(F))$ such that there exists $w \in \phi^{-1}(u)$ with $vw \in E(W)$. We call $W$ an \emph{extender} of $F$, and $W[\phi(F)]$ the \emph{source} of the extension.
\end{defs}

Note the source is a subgraph of $H$, and that it is not necessarily induced. 

In order to establish certain structural properties of $C_{2t+1}$-critical graphs in the subsequent section, we will make extensive use of the following lemma. 

\begin{lemma}\label{potential-extension}
Let $G$ be an $H$-critical graph with potential $p(G)$. Let $F$ be a proper subgraph of $G$ that is not isomorphic to $H$. If $F'$ is an extension of $F$ with extender $W$ and source $X$, then 

\begin{equation*}
    p(F') = p(F) + p(W) - p(X).
\end{equation*}
\end{lemma}
\begin{proof}
By the definitions of $F'$, $W'$, and $X$ given in \ref{potdefs}, $v(F') = v(F) + v(W) - v(X)$. Furthermore, $e(F') = e(F)+ e(W)- e(X)$. Therefore since $p(F') = \alpha v(F')- \beta e(F')$ and both $\alpha$ and $\beta$ are greater than 0, $p(F')= p(F) + p(W) - p(X)$.
\end{proof}

\subsection{Outline of the Proof of Theorem \ref{main}}
The proof of Theorem \ref{main} is obtained via reducible configurations and discharging. The structural results of a minimum counterexample to Theorem \ref{main} are rather lengthy, and are found in Section 3. The discharging portion is found in Section 4; it is a bit unusual, as we will discuss later. More specifically, the discharging proceeds as follows.

We let $G$ be a counterexample to Theorem \ref{main} with $v(G)$ minimum and, subject to that, with $e(G)$ minimum. Since potential is integral, it follows that $p(G) \geq 3$. We will assign initial charge of $ch_0(v) = 15\deg(v)-2\textnormal{wt}(v)-34$ to each vertex $v \in V(G)$ with $\deg(v) \geq 3$, and $ch_0(v) = 0$ to each vertex of degree two. Thus $\sum_{v \in V(G)} ch_0(v) = -2p(G) \leq -6$. We aim to show that after discharging, every structure in the graph has non-negative charge, thus arriving at a contradiction. 

Initially, the only structures with negative charge are the degree three vertices with weight at least six.  Using Lemma \ref{max-str} and \ref{weight}, we rule out several such types of vertices with high weight. Using Lemmas \ref{(4,4,k)}, \ref{4,3,k}, and \ref{3,3,2} we will rule out the existence of vertices of type $(3,3,2)$, $(4,4,k)$ where $k \geq 0$, and $(4,3,k)$ where $k \geq 1$. This is accomplished by characterizing the intersection of 7- and 9-cycles in Lemmas \ref{7cycles}, \ref{7and9cycles}, and \ref{9cycles}. As not all degree three vertices with weight at least six can be ruled out, the remainder of the lemmas in Section 3 aim to establish the local structure surrounding the vertices of degree three and the cells and vertices that will later send them charge. In particular, Lemmas \ref{4str}, \ref{3str}, and \ref{deg4s} will establish the neighbouring structure of vertices incident with long strings. In Lemma \ref{deg3cells}, we show that $G$ does not contain cells of low degree. Finally, Lemmas \ref{3,2,2s}, \ref{2,2,2s}, and \ref{(3,3,0)} establish the neighbouring structure of certain types of vertices not contained in cells.  
 
The discharging itself begins with each cell collecting all of the charge of nearby vertices. Next, vertices of degree at least four send charge to the vertices of degree three that were not near cells. Finally, vertices of degree three and low weight send charge to nearby vertices of degree three and weight at least six.

The discharging process used is a bit unusual. For one thing, the rules are performed sequentially. This ensures that in the later steps of the discharging process, much of the local structure surrounding the vertices and cells receiving charge will be known. Finally, the vertices and cells in $G$ only send charge along short strings. This way, the initial charges ensure that if a vertex or cell sends charge to many structures, it follows that the vertex or cell sending charge has relatively low weight and so consequently has a large amount of charge to spare. We believe a similar form of sequential discharging may prove useful in proving Conjecture \ref{jaeger}.
\end{section}
\section{Structure of $C_7$-Critical Graphs}

Our main result is the following:

\main*

As before, we will restate this theorem in terms of potentials. In this section, the \textit{potential}\footnote{In Sections 3 and 4, \emph{potential} denotes the $(17,15)$-potential.} of a graph $G$ will be defined as $p(G) = 17v(G)-15e(G)$. We aim to prove that if $G \not \in \{C_3, C_5\}$ is a $C_7$-critical graph, then  $p(G) \leq 2$. We note this bound is tight: examples of $C_7$-critical graphs with potential 2 can be found in Figure \ref{fig:tight}.

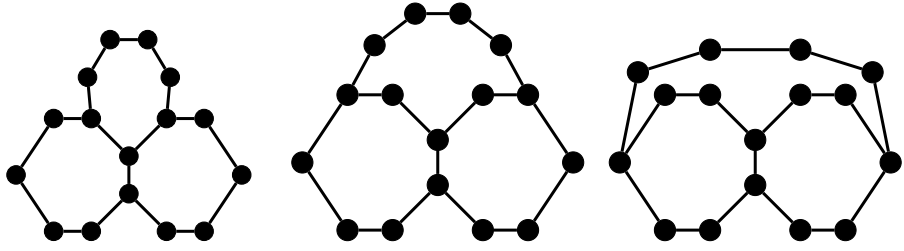
\begin{figure}
	\begin{center}
        
            \begin{tikzpicture}[scale=0.5]	
			\tikzset{black node/.style={shape=circle,draw=black,fill=black,inner sep=0pt, minimum size=7pt}}
            \tikzset{white node/.style={shape=circle,draw=black,fill=white,inner sep=0pt, minimum size=8pt}}
			\tikzset{invisible node/.style={shape=circle,draw=black,fill=black,inner sep=0pt, minimum size=1pt}}
				\tikzset{edge/.style={black, line width=.4mm}}		

                \node[black node] (1) at (0,0){};
                \node[black node] (2) at (0,-1){};
                \node[black node] (3) at (1,-2){};                
                \node[black node] (4) at (2,-2){};                
                \node[black node] (5) at (3,-0.5){};                
                \node[black node] (6) at (2,1){};                
                \node[black node] (7) at (1,1){};
                  
                \node[black node] (8) at (-1,-2){};
                \node[black node] (9) at (-2,-2){};
                \node[black node] (10) at (-3,-0.5){};                
                \node[black node] (11) at (-2,1){};                
                \node[black node] (12) at (-1,1){};                
                
                \node[black node] (13) at (-1.1,2.1){};                
                \node[black node] (14) at (-0.5,3.1){};   
                \node[black node] (15) at (0.5,3.1){};                
                \node[black node] (16) at (1.1,2.1){};
                
                \draw[edge] (1)--(2); 
                \draw[edge] (2)--(3); 
                \draw[edge] (3)--(4); 
                \draw[edge] (4)--(5); 
                \draw[edge] (5)--(6); 
                \draw[edge] (6)--(7); 
                \draw[edge] (7)--(1); 
                \draw[edge] (2)--(8); 
                \draw[edge] (8)--(9); 
                \draw[edge] (9)--(10); 
                \draw[edge] (10)--(11); 
                \draw[edge] (11)--(12); 
                \draw[edge] (12)--(1); 
                \draw[edge] (12)--(13); 
                \draw[edge] (13)--(14); 
                \draw[edge] (14)--(15); 
                \draw[edge] (15)--(16); 
                \draw[edge] (16)--(7);                 
			\end{tikzpicture}
            \hskip 4mm
                        \begin{tikzpicture}[scale=0.6]	
			\tikzset{black node/.style={shape=circle,draw=black,fill=black,inner sep=0pt, minimum size=8pt}}
            \tikzset{white node/.style={shape=circle,draw=black,fill=white,inner sep=0pt, minimum size=8pt}}
			\tikzset{invisible node/.style={shape=circle,draw=black,fill=black,inner sep=0pt, minimum size=1pt}}
				\tikzset{edge/.style={black, line width=.4mm}}		

                \node[black node] (1) at (0,0){};
                \node[black node] (2) at (0,-1){};
                \node[black node] (3) at (1,-2){};                
                \node[black node] (4) at (2,-2){};                
                \node[black node] (5) at (3,-0.5){};                
                \node[black node] (6) at (2,1){};                
                \node[black node] (7) at (1,1){};
                  
                \node[black node] (8) at (-1,-2){};
                \node[black node] (9) at (-2,-2){};
                \node[black node] (10) at (-3,-0.5){};                
                \node[black node] (11) at (-2,1){};                
                \node[black node] (12) at (-1,1){};                
                
                \node[black node] (13) at (-1.4,2.1){};                
                \node[black node] (14) at (-0.5,2.8){};   
                \node[black node] (15) at (0.5,2.8){};                
                \node[black node] (16) at (1.4,2.1){};
                
                \draw[edge] (1)--(2); 
                \draw[edge] (2)--(3); 
                \draw[edge] (3)--(4); 
                \draw[edge] (4)--(5); 
                \draw[edge] (5)--(6); 
                \draw[edge] (6)--(7); 
                \draw[edge] (7)--(1); 
                \draw[edge] (2)--(8); 
                \draw[edge] (8)--(9); 
                \draw[edge] (9)--(10); 
                \draw[edge] (10)--(11); 
                \draw[edge] (11)--(12); 
                \draw[edge] (12)--(1); 
                \draw[edge] (11)--(13); 
                \draw[edge] (13)--(14); 
                \draw[edge] (14)--(15); 
                \draw[edge] (15)--(16); 
                \draw[edge] (16)--(6);                 
			\end{tikzpicture}
            \hskip 2mm
             \begin{tikzpicture}[scale=0.6]	
			\tikzset{black node/.style={shape=circle,draw=black,fill=black,inner sep=0pt, minimum size=8pt}}
            \tikzset{white node/.style={shape=circle,draw=black,fill=white,inner sep=0pt, minimum size=8pt}}
			\tikzset{invisible node/.style={shape=circle,draw=black,fill=black,inner sep=0pt, minimum size=1pt}}
				\tikzset{edge/.style={black, line width=.4mm}}		

                \node[black node] (1) at (0,0){};
                \node[black node] (2) at (0,-1){};
                \node[black node] (3) at (1,-2){};                
                \node[black node] (4) at (2,-2){};                
                \node[black node] (5) at (3,-0.5){};                
                \node[black node] (6) at (2,1){};                
                \node[black node] (7) at (1,1){};
                  
                \node[black node] (8) at (-1,-2){};
                \node[black node] (9) at (-2,-2){};
                \node[black node] (10) at (-3,-0.5){};                
                \node[black node] (11) at (-2,1){};                
                \node[black node] (12) at (-1,1){};                
                
                \node[black node] (13) at (-2.6,1.5){};                
                \node[black node] (14) at (-1,2){};   
                \node[black node] (15) at (1,2){};                
                \node[black node] (16) at (2.6,1.5){};
                
                \draw[edge] (1)--(2); 
                \draw[edge] (2)--(3); 
                \draw[edge] (3)--(4); 
                \draw[edge] (4)--(5); 
                \draw[edge] (5)--(6); 
                \draw[edge] (6)--(7); 
                \draw[edge] (7)--(1); 
                \draw[edge] (2)--(8); 
                \draw[edge] (8)--(9); 
                \draw[edge] (9)--(10); 
                \draw[edge] (10)--(11); 
                \draw[edge] (11)--(12); 
                \draw[edge] (12)--(1); 
                \draw[edge] (10)--(13); 
                \draw[edge] (13)--(14); 
                \draw[edge] (14)--(15); 
                \draw[edge] (15)--(16); 
                \draw[edge] (16)--(5);                 
			\end{tikzpicture}

    \end{center}
	\caption{Examples of $C_7$-critical graphs with potential 2.}
	\label{fig:tight}
\end{figure}

Since potential is integral, a counterexample to Theorem \ref{main} has potential at least $3$.  A \textit{minimum counterexample} to Theorem \ref{main} is a $C_7$-critical graph $G \not \in \{C_3, C_5\}$ with $p(G) \geq 3$, minimal with respect to $v(G)$, and, subject to that, with respect to $e(G)$.

For the remainder of Sections 3 and 4, $G$ will be a minimum counterexample to Theorem \ref{main}. The following two subsections concern the structure of $G$. Section \ref{struc} contains general structural results, and Section \ref{forb} rules out the presence of certain substructures in $G$. 

\subsection{General Structure Results}\label{struc}

The lemmas in this section provide us with a general framework for $G$. Lemma \ref{potentials7} concerns the potential of subgraphs of $G$, and will be useful in proving further structural lemmas.  Lemma \ref{girth} establishes a lower bound for the girth of $G$. The proofs of Lemmas \ref{4str}, \ref{3str}, and \ref{deg4s} will establish the neighbouring structure of vertices incident with long strings.  Finally, with Lemmas \ref{7cycles}, \ref{7and9cycles}, and \ref{9cycles} we will characterize the intersections of distinct 7-cycles and 9-cycles in $G$. We require the following definition. 

\begin{defs}\label{pth}
Let $H$ be a graph. We denote by $P_t(H)$ the set of graphs obtained from $H$ by adding a path $P$ of length $t$ joining two distinct vertices of $H$, such that the internal vertices of $P$ are disjoint from $V(H)$. 
\end{defs}

Our first structural result in this section concerns the density of subgraphs of $G$.
\begin{lemma}\label{potentials7} Let $H$ be a subgraph of $G$.  Then the following all hold:
\begin{enumerate}[label=(\roman*)]
\item $p(H) \geq 3$ if $H = G$,
\item $p(H) \geq 20-2k$ if $G \in P_k(H)$ and $k \in \{3,4,5\}$,
\item $p(H) = 14$ if $H = C_7$, and
\item $p(H) \geq 15$ otherwise.  
\end{enumerate}
\end{lemma}

\begin{proof}
Suppose not. Let $H$ be a counterexample to Lemma \ref{potentials7}, maximal with respect to $v(H)$, and subject to that, with respect to $e(H)$. Since $G$ is a minimum counterexample to Theorem \ref{main} and potential is integral, if $H = G$, then (i) holds\textemdash a contradiction.  If $H$ is isomorphic to $C_7$, then (iii) holds, a contradiction.  We may therefore assume $H \not \in \{C_7, G\}$. 

First suppose that $H$ is not induced. Then $p(G[V(H)]) = p(H) - 15(e(G[V(H)])-e(H)) $. As $H$ is not induced, $e(G[V(H)]) - e(H) \geq 1$ and so it follows that $p(G[V(H)]) \leq p(H) - 15 \leq -1$. But then $G[V(H)]$ is a counterexample to Lemma \ref{potentials7}, contradicting our choice of $H$.

We may therefore assume $H$ is induced. Note every proper subgraph $H$ of $C_7$ has potential at least $17$, since $p(H) = 2t + 17$ if $H$ is a path with $t$ edges. Thus if $H$ is a proper subgraph of $C_7$, (iv) holds, a contradiction.  Since $G$ is $C_7$-critical and $H \subsetneq G$, $H$ has a homomorphism $\phi$ to a subgraph of $C_7$. Let $H'$ be an extension of $H$ with extender $W$ and source $X$ (see Definition \ref{potdefs}). By Lemma \ref{potential-extension}, $p(H') = p(H)+p(W)-p(X)$.

Suppose first that $W$ is a triangle. Since $W \not \subset G$, $W$ contains at least one vertex in $\phi(H)$ (see Definition \ref{phi(F)}). Similarly, since $W \not \subset \phi(H)$, $W$ contains at least one vertex in $V(G) \setminus V(H)$. This gives rise to at least two edges in $E(H')$ from vertices in $V(G)\setminus V(H)$ to $V(H)$. Thus $G$ has a path $P_k$, $k \in \{2, 3\}$, with endpoints in $H$ and internal vertices in $G \setminus V(H)$ (see Figure \ref{fig:5tri}).  Since $e(H \cup P_k)-e(H) = k$ and $v(H\cup P_k) -v(H) = k-1$, we have that $p(H \cup P_k) = p(H) + 17(k-1) - 15k$, and so $p(H\cup P_k) = p(H) + 2k - 17 \leq 2k-3$ (since $p(H) \leq 14$). Since $H$ is maximal, $H \cup P_k$ is not a counterexample to Lemma \ref{potentials7}. Thus $p(H \cup P_k) \geq 3$ and so $k = 3$. As $p(H\cup P_3) \leq 3$, we have $G = H \cup P_3$. But then $3 \leq p(H \cup P_3) = p(H) - 11$, and so $p(H) \geq 14$. But now $G = (P_3 \cup H) \in P_3(H)$ and so (ii) holds \textemdash a contradiction.

\begin{figure}
\begin{center}

    \begin{tikzpicture}[scale=0.8]	
			\tikzset{black node/.style={shape=circle,draw=black,fill=black,inner sep=0pt, minimum size=5pt}}
            \tikzset{white node/.style={shape=circle,draw=black,fill=white,inner sep=0pt, minimum size=11pt}}
			\tikzset{invisible node/.style={shape=circle,draw=black,fill=black,inner sep=0pt, minimum size=1pt}}
				\tikzset{edge/.style={black, line width=.4mm}}		

                \draw[fill=gray!30] (-3,0) ellipse (2cm and 1cm);
                \node[black node] (3) at (-3.5,0){};     
                \node[black node] (4) at (-3,1.5){};
                \node[black node] (5) at (-2.5,0){};
                
                \draw[fill=gray!30] (2.75,0) ellipse (2cm and 1cm);
                \node[black node] (8) at (2,0){};     
                \node[black node] (9) at (2.5,1.5){};
                \node[black node] (10) at (3,1.5){};
                \node[black node] (11) at (3.5,0){};

                \draw[edge] (3)--(4);
                \draw[edge] (4)--(5);
                \draw[edge] (8)--(9);
                \draw[edge] (9)--(10);
                \draw[edge] (10)--(11);
                
                \node[] at (-3, -0.5) {$H$};
                \node[] at (2.75,-0.4) {$H$};

			\end{tikzpicture}  
        \end{center}
	\caption{$G$ contains a path $P$ of length either 2 or 3, such that the endpoints of $P$ are in $H$ and the internal vertices of $P$ are in $G \setminus V(H)$.}
	\label{fig:5tri}
    
\end{figure}
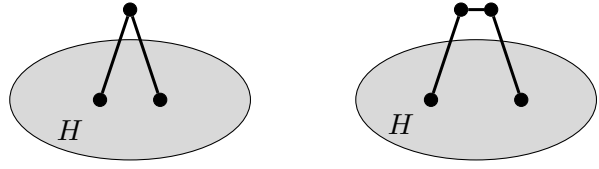

Suppose next that $W$ is a 5-cycle. Since $W \not \subset G$, $W$ contains at least one vertex in $\phi(H)$. Note there are at most two components in $W \setminus V(H)$, since each such component gives rise to at least two edges in $W$ and $e(W) = 5$.  Thus each component in $W \setminus V(H)$ is a path, and so at least one of the following cases hold (see Figure \ref{fig:5cyc}):

\renewcommand\labelenumi{(\theenumi)}
\begin{enumerate}
\item $G$ contains a path $P_k$, $k \in \{2, 3, 4, 5\}$, joining two distinct vertices of $H$ such that the internal vertices of $P_k$ are not in $H$, or 
\item $G$ contains a path $P_2$ joining two distinct vertices of $H$ such that the internal vertex of $P_2$ is not in $H$, and a second path $Q_k$, $k \in \{2,3\}$, joining distinct vertices of $H$ such that the internal vertices of $Q_k$ are not in $H$.
\end{enumerate} 

\begin{figure}
\begin{center}
    \begin{tikzpicture}[scale=0.7]	
			\tikzset{black node/.style={shape=circle,draw=black,fill=black,inner sep=0pt, minimum size=5pt}}
            \tikzset{white node/.style={shape=circle,draw=black,fill=white,inner sep=0pt, minimum size=11pt}}
			\tikzset{invisible node/.style={shape=circle,draw=black,fill=black,inner sep=0pt, minimum size=1pt}}
				\tikzset{edge/.style={black, line width=.4mm}}		

                \draw[fill=gray!30] (-4,0) ellipse (1.5cm and 0.8cm);
                \node[black node] (3) at (-4.5,0){};     
                \node[black node] (4) at (-4,1.5){};
                \node[black node] (5) at (-3.5,0){};
                
                \draw[fill=gray!30] (-0.5,0) ellipse (1.5cm and 0.8cm);
                \node[black node] (8) at (-1.5,0){};     
                \node[black node] (9) at (-0.75,1.5){};
                \node[black node] (10) at (-0.25,1.5){};
                \node[black node] (11) at (0.5,0){};
                
                 \draw[fill=gray!30] (3,0) ellipse (1.5cm and 0.8cm);
                \node[black node] (12) at (2,0){};     
                \node[black node] (13) at (2.5,1.5){};
                \node[black node] (14) at (3,1.5){};
                \node[black node] (15) at (3.5,1.5){};
                \node[black node] (16) at (4,0){};

                \draw[fill=gray!30] (6.5,0) ellipse (1.5cm and 0.8cm);
                \node[black node] (17) at (5.5,0){};     
                \node[black node] (18) at (5.75,1.5){};
                \node[black node] (19) at (6.25,1.5){};
                \node[black node] (20) at (6.75,1.5){};
                \node[black node] (21) at (7.25,1.5){};
                \node[black node] (22) at (7.5,0){};

                \draw[edge] (3)--(4);
                \draw[edge] (4)--(5);
                
                \draw[edge] (8)--(9);
                \draw[edge] (9)--(10);
                \draw[edge] (10)--(11);
                
                \draw[edge] (12)--(13);
                \draw[edge] (13)--(14);
                \draw[edge] (14)--(15);
                \draw[edge] (15)--(16);
                
                \draw[edge] (17)--(18);
                \draw[edge] (18)--(19);
                \draw[edge] (19)--(20);
                \draw[edge] (20)--(21);
                \draw[edge] (21)--(22);
                
                \node[] at (-4, -0.4) {$H$};
                \node[] at (-0.5,-0.4) {$H$};
                \node[] at (3,-0.4) {$H$};
                \node[] at (6.5,-0.4) {$H$};
			\end{tikzpicture}  
			\hskip 2mm
            \begin{tikzpicture}[scale=0.8]	
			\tikzset{black node/.style={shape=circle,draw=black,fill=black,inner sep=0pt, minimum size=5pt}}
            \tikzset{white node/.style={shape=circle,draw=black,fill=white,inner sep=0pt, minimum size=11pt}}
			\tikzset{invisible node/.style={shape=circle,draw=black,fill=black,inner sep=0pt, minimum size=1pt}}
				\tikzset{edge/.style={black, line width=.4mm}}

                 \draw[fill=gray!30] (2,0) ellipse (2cm and 0.8cm);
                \node[black node] (1) at (0.75,0){};     
                \node[black node] (2) at (1.25,1.5){};
                \node[black node] (3) at (1.75,0){};
                \node[black node] (4) at (2.25,0){};
                \node[black node] (5) at (2.75,1.5){};
                \node[black node] (6) at (3.25,0){};

                \draw[fill=gray!30] (6.5,0) ellipse (2cm and 0.8cm);
                \node[black node] (7) at (5,0){};     
                \node[black node] (8) at (5.5,1.5){};
                \node[black node] (9) at (6,0){};
                \node[black node] (10) at (6.5,0){};
                \node[black node] (11) at (7,1.5){};
                \node[black node] (12) at (7.5,1.5){};
                \node[black node] (13) at (8,0){};
                
                \draw[edge] (1)--(2);
                \draw[edge] (2)--(3);
                \draw[edge] (4)--(5);
                \draw[edge] (5)--(6);
                
                \draw[edge] (7)--(8);
                \draw[edge] (8)--(9);
                \draw[edge] (10)--(11);
                \draw[edge] (11)--(12);
                \draw[edge] (12)--(13);

                \node[] at (2,-0.4) {$H$};
                \node[] at (6.5,-0.4) {$H$};
			\end{tikzpicture}  
        \end{center}
	\caption{$G$ contains a path $P_k$, $k \in \{2, 3, 4, 5\}$, joining two distinct vertices of $H$ such that the internal vertices of $P_k$ are not in $H$, or $G$ contains a path $P_2$ joining two distinct vertices of $H$ such that the internal vertex of $P_2$ is not in $H$, and a second path $Q_k$, $k \in \{2,3\}$, joining distinct vertices of $H$ such that the internal vertices of $Q_k$ are not in $H$.}
	\label{fig:5cyc}
    
\end{figure}
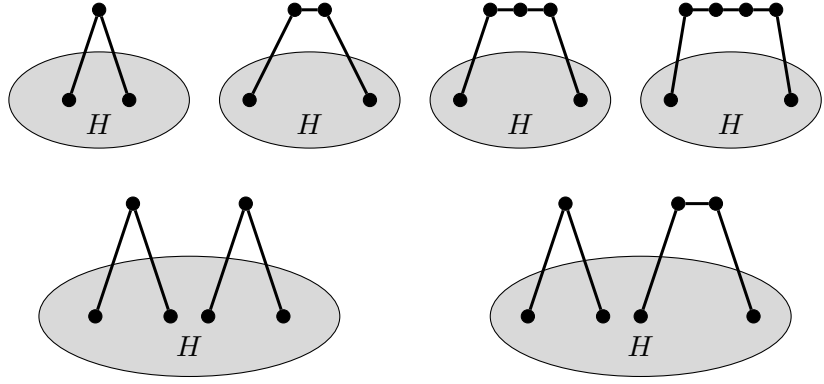

Suppose that $G$ contains a path $P_2$ as described in (1) or (2). Note that $p(H \cup P_2) = p(H) -13 \leq 1$. But since $H$ is maximal this contradicts our choice of $H$.  If $G$ does not contain a path $P_2$ as described, then it contains a path $P_k$, with  $k \in \{3,4,5\}$ as described in (1). We have $p(H \cup P_k) = p(H) + 17(k-1)-15k = p(H) +2k-17$.  Since $p(H) \leq 14$, it follows that $p(H \cup P_k) \leq 2k-3$. Since $H \cup P_k$ is not a counterexample, $3 \leq p(H\cup P_k)$. As $3 \leq 2k-3$, it follows that $k = 3$ and $p(H) = 14$. But then (ii) holds \textemdash a contradiction.

We can therefore assume $W \not \in \{C_3, C_5\}$. Recall $H'$ is an extension of $H$ with extender $W$ and source $X$. Note since $H \not \subseteq C_7$, we have that $v(\phi(H)) < v(H)$, and consequently that $v(W) < v(G)$. Since $G$ is a minimum counterexample to Theorem \ref{main}, and $W$ is neither a 3-cycle nor a 5-cycle, it follows that $p(W) \leq 2$. Since $X$ is a subgraph of a 7-cycle, $p(X) \geq 14$. We have therefore that $p(H') = p(W) + p(H) - p(X) \leq 2 + p(H) - 14$. Since $p(H) \leq 14$ we have $p(H') \leq 2$. But then $H'$ is a counterexample to Lemma \ref{potentials7}, contradicting our choice of $H$. 
\end{proof}

The following lemma gives a lower bound for the girth of $G$.
\begin{lemma}\label{girth}
$G$ has girth at least 7.
\end{lemma}
\begin{proof}
Suppose not. Note $G$ does not contain a 5-cycle or a triangle, as these are $C_7$-critical themselves and $G \not \in \{C_3, C_5\}$. It follows that $G$ contains a cycle $C$ of length $2t$, where $t \in \{2,3\}$. But then $p(C) = 17(2t)-15(2t) = 2t \leq 12$, so $G \in \{P_4(C), P_5(C)\}$ by lemma \ref{potentials7}. But then $G$ is a theta graph, and no such graph is $C_{2t+1}$-critical \textemdash a contradiction. 
\end{proof}

The following three lemmas are used to show that certain types of vertices of degree three and four are contained either in cells or in cycles of length nine. 

\begin{lemma}\label{4str}
If $v$ is a vertex of degree three incident with strings $S_1$, $S_2$, and $S_3$ such that $S_3$ is a 4-string, then $S_1 \cup S_2$ is contained in a cell.
\end{lemma}

\begin{proof}
Let $\{a_1, a_2\} = N(v) \setminus V(S_3)$. It suffices to show that the path $a_1va_2$ is contained in a cell, since the internal vertices of $S_1$ and $S_2$ (if they exist) have degree 2. Suppose this is not the case. Let $G'$ be the graph obtained from $G$ by identifying $a_1$ and $a_2$ to a new vertex $z$. Note since $a_1va_2$ is not contained in a cell and $G$ has girth at least 7 by Lemma \ref{girth}, $G'$ contains no triangle nor 5-cycle. Let $x \neq v$ be an endpoint of $S_3$ in $G$, and let $S = S_3 - x$.  If $G'$ admits a homomorphism $\phi$ to $C_7$, then $\phi$ extends to $G$ by setting $\phi(a_1) = \phi(a_2) = \phi(z)$. Thus there does not exist a homomorphism of $G'$ to $C_7$, and so $G'$ contains a $C_7$-critical subgraph $G''$. Since $G'' \not \in \{C_3, C_5\}$, since $v(G'') < v(G)$, and since $G$ is a minimum counterexample, we have that $p(G'') \leq 2$. Since $G'' \not \subseteq G$, it follows that $z \in V(G'')$. Furthermore, $S$ is not contained in $G''$, since by Lemma \ref{max-str} $G''$ does not contain the 5-string $S_3z$ and the minimum degree of $G''$ is at least 2. Let $F$ be the graph obtained from $G''$ by splitting $z$ back into vertices $a_1$ and $a_2$, and adding the path $a_1va_2$. The potential of $F$ is given by $p(F) = p(G'') + 17(2) - 15(2) \leq 6$. By Lemma \ref{potentials7}, $F = G$. But this is a contradiction, since $S-v$ is not contained in $F$.
\end{proof}


\begin{lemma}\label{3str}
If $v$ is a vertex of degree three incident with strings $S_1$, $S_2$, and $S_3$ such that $S_3$ is a 3-string and both $S_1$ and $S_2$ contain at least two edges, then $S_1 \cup S_2$ is contained in a cell or a 9-cycle.
\end{lemma}
\begin{proof}

Let $\{a_1, b_1\} = N(v) \setminus V(S_3)$. It suffices to show the path $a_1vb_1$ is contained in a cell or 9-cycle, since the internal vertices of $S_1$ and $S_2$ have degree 2. Suppose this is not the case. Let $a_2 = N(a_1) - v$, and let $b_2 = N(b_1)-v$. Let $G'$ be the graph obtained from $G$ by identifying $a_1$ and $b_1$ to a new vertex $z_1$, and identifying $a_2$ and $b_2$ to a new vertex $z_2$. Note since $a_1vb_1$ is not contained in a cell or 9-cycle, $G'$ contains no 3- nor 5-cycle. Let $x \neq v$ be an endpoint of $S_3$ in $G$, and let $S = (S_3\cup z_1) \setminus \{x\}$. If $G'$ admits a homomorphism $\phi$ to $C_7$, then $\phi$ extends to $G$ by setting $\phi(a_1) = \phi(b_1) = \phi(z_1)$ and $\phi(a_2) = \phi(b_2) = \phi(z_2)$. Thus there does not exist a homomorphism of $G'$ to $C_7$, and so $G'$ contains a $C_7$-critical subgraph $G''$. Since $G'' \not \in \{C_3, C_5\}$, since $v(G'') < v(G)$, and since $G$ is a minimum counterexample, it follows that $p(G'') \leq 2$. Furthermore, $S$ is not contained in $G''$ since by Lemma \ref{max-str} $G''$ does not contain a 5-string $S_3z_1z_2$ and the minimum degree of $G''$ is at least 2. Since $G'' \not \subseteq G$ and $z_1 \not \in V(G'')$, we have that $z_2 \in V(G'')$. 

Let $F$ be the graph obtained from $G$ by splitting $z_1$ and $z_2$ back into $a_1, b_1$ and $a_2, b_2$, respectively, and adding the path $b_2b_1va_1a_2$. The potential of $F$ is given by $p(F) = p(G'') + 17(4)-15(4) \leq 10$. By Lemma \ref{potentials7}, either $F = G$ or $G \in P_5(F)$. Since $S_3 \setminus \{x, v\}$ is not contained in $F$, we have that $F \neq G$ and so $G = F \cup P$ for a path $P$ of length 5. But again since $S_3 \setminus \{x,v\}$ is not contained in $F$, $P$ contains $S_3$. But by definition, $S_3$ is a path of length 4 with endpoints of degree three \textemdash a contradiction.
\end{proof}

\begin{lemma}\label{deg4s}
If $v$ is a vertex of degree 4 incident with strings $S_1$, $S_2$, $S_3$, and $S_4$ such that $S_4$ is a 4-string, then there exists $\{i, j\} \subset \{1,2,3\}$ with $i \neq j$ such that $S_i \cup S_j$ is contained in a cell. 

\end{lemma}
\begin{proof}
Let $\{u_1, u_2, u_3\} = N(v) \setminus V(S_4)$.  It suffices to show one of the paths $u_1vu_2$, $u_1vu_3$ or $u_2vu_3$ is contained in a cell, since the internal vertices of $S_1$, $S_2$ and $S_3$ (if they exist) have degree 2. Suppose this is not the case. Let $G'$ be the graph obtained from $G$ by identifying $u_1, u_2$ and $u_3$ to a new vertex $z$. Note $G'$ contains no 3- nor 5-cycle. Let $x \neq v$ be an endpoint of $S_4$ in $G$, and let $S = S_4-x$. If $G'$ admits a homomorphism $\phi$ to $C_7$, then $\phi$ extends to $G$ by setting $\phi(u_1) = \phi(u_2) = \phi(u_3) = \phi(z)$. Thus there does not exist a homomorphism of $G'$ to $C_7$, and so $G'$ contains a $C_7$-critical subgraph $G''$. Since $G'' \not \in \{C_3, C_5\}$, since $v(G'') < v(G)$, and since $G$ is a minimum counterexample, it follows that $p(G'') \leq 2$. Since $G'' \not \subseteq G$, we have $z \in V(G'')$. Furthermore, $S$ is not contained in $G''$, since by Lemma \ref{max-str} $G''$ does not contain a 5-string and the minimum degree of $G''$ is at least 2. Let $F$ be the graph obtained from $G''$ by splitting $z$ back into vertices $u_1, u_2$ and $u_3$, and adding the path $u_1vu_2$ and the edge $vu_3$. The potential of $F$ is given by $p(F) = p(G'') + 17(3) - 15(3) \leq 8$. By Lemma \ref{potentials7}, $F = G$. But this is a contradiction, since $S-v$ is not contained in $F$.
\end{proof}

Finally, the last three lemmas in Section 4.1 characterize the intersection of distinct 7- and 9-cycles in $G$.  Together with Lemmas \ref{4str}, \ref{3str}, and \ref{deg4s}, the following lemmas will allow us to rule out the existence of certain types of vertices in Section 4.2. 

\begin{lemma}\label{7cycles} Let $C$ and $C'$ be distinct 7-cycles in $G$. Then $C$ and $C'$ are vertex-disjoint. 
\end{lemma}
\begin{proof}Suppose not, and let $H = C \cup C'$.  First suppose $C \cap C'$ has two components $P$ and $P'$. Since paths have potential at least 17, $p(H) \leq p(C) + p(C') - p(P) - p(P') \leq 14 + 14 -17 - 17 \leq -6$. This contradicts Lemma \ref{potentials7}. 

Thus we may assume the cycles $C$ and $C'$ intersect in a single path $P$ of length $k$. Note $0 \leq k \leq 3$, as otherwise $H$ (and thus $G$) contains a cycle of length at most six, contradicting Lemma \ref{girth}. The potential of $H$ is given by $p(H) = p(C) + p(C') - p(P) =14 + 14 -(2k+17) = 11-2k$. Note $H \neq G$ since no theta graph is $C_7$-critical. By Lemma \ref{potentials7}, we have that $p(H) \geq 10$. Since $p(H) \leq 11$ we have $k=0$, and by Lemma \ref{potentials7} $G \in P_5(H)$. Since $k=0$, $P$ is a single vertex $v$. Let $Q$ be the path of length five such that $G = H \cup Q$. Since $G$ is $C_7$-critical, $Q \cap (C-v) \neq \emptyset$ and $Q \cap (C'- v) \neq \emptyset$ as otherwise $v$ is a cut vertex, and no $C_7$-critical graph contains a cut-vertex by Lemma \ref{cut}.  The cell $C$ has degree three, and is incident with strings $S_1$, $S_2$ and $Q$, where $S_1 \cup S_2 = C'$. Note $Q$ is a 4-string, and $S_1$ and $S_2$ together contribute 5 to the weight of $C$. Thus $\textrm{wt}(C) = 9$. But this is a contradiction, since by Lemma \ref{cell-weight} $\textrm{wt}(C) \leq 8$. 
\end{proof}


\begin{lemma}\label{7and9cycles}
Let $C$ and $C'$ be cycles of length seven and nine, respectively, in $G$. Then $C$ and $C'$ are edge-disjoint. 
\end{lemma}
\begin{proof}
Suppose not. Let $C$ and $C'$ be the cycles of length 7 and 9, respectively, of $G$, chosen such that their intersection is maximal. Let $H = C \cup C'$. First suppose $C \cap C'$ has at least two components $P$ and $P'$. Since paths have potential at least 17, $p(H) \leq p(C) + p(C') - p(P) - p(P') \leq 14+18-17-17 = -2$, contradicting Lemma \ref{potentials7}.

Thus we may assume $C \cap C'$ is a single path $P$ of length $k$. Note since $C$ and $C'$ share an edge, $k \geq 1$. Furthermore, $k \leq 6$, as otherwise $H$ (and thus $G$) contains a cycle of length at most six, contrary to Lemma \ref{girth}. The potential of $H$ is therefore given by $p(H) = p(C) + p(C') - p(P) =18 + 14 -(2k+17) =15-2k$. Note $H \neq G$, since $H$ is a theta graph and no such graph is $C_7$-critical. By Lemma \ref{potentials7}, we have that $p(H) \geq 10$. Since $p(H) \leq 15$, we have $k \in \{1, 2\}$.   We now break into cases according to the value of $k$.
\vskip 4mm
\noindent{
\emph{Case 1: $k=1$.}} By Lemma \ref{potentials7}, $G \in P_4(H) \cup P_5(H)$. Let $Q_i$ be the path of length $i \in \{4,5\}$ with $G = H \cup Q_i$.  Note since $G$ is $C_7$-critical, $Q_i \cap P  = \emptyset$. To see this, suppose not. Then at least one of $(C\setminus V(P)) \cap Q_i$ and $(C' \setminus V(P)) \cap Q_i$ is the empty set. But then $C$ or $C'$ contains a $j$-string with $j \geq 5$, contradicting Lemma \ref{max-str}. Thus $Q_i \cap P = \emptyset$ and both $Q_i \cap C \neq \emptyset$ and $Q_i \cap C' \neq \emptyset$. The cycle $C$ has degree three, and its incident strings are $S_1$, $S_2$, and $Q_i$, where $S_1 \cup S_2 \cup P = C'$. $Q_i$ contributes $i-1$ to the weight of $C$, and $S_1$ and $S_2$ together contribute 6. Thus $\textrm{wt}(C) = i + 5$. Since $i \geq 4$, we have that $\textrm{wt}(C) \geq 9$, contradicting Lemma \ref{cell-weight}. 
\vskip 4mm

\noindent{\emph{Case 2: $k=2$.}} By Lemma \ref{potentials7}, $G \in P_5(H)$. Let $Q$ be the path of length five with $G = H \cup Q$. Note $(C' \setminus V(P)) \cap Q \neq \emptyset$ as otherwise $C'$ contains a 6-string, contrary to Lemma \ref{max-str}. Note therefore at least one of the endpoints of $P$ is not in $V(Q)$. Let $v$ be an endpoint of $P$ with $v \not \in V(Q)$. Let $u \in V(C')\setminus V(C)$ be adjacent to $v$. Suppose $(C \setminus V(P)) \cap Q = \emptyset$. Then $v$ is incident with a 4-string of $G$ contained in $C$, and so by Lemma \ref{4str}, $Pu$ is contained in a cell $C'' \neq C$. Since $P \subset C$, this contradicts Lemma \ref{7cycles}, since distinct 7-cycles in $G$ are vertex-disjoint. Thus we may assume $Q \cap P = \emptyset$, and both $Q \cap C \neq \emptyset$ and $Q \cap C' \neq \emptyset$. Let $w$ be the endpoint of $Q$ contained in $C'$. By Lemma \ref{4str}, since $Q$ is a 4-string, the path formed by $w$ and $w$'s neighbours in $C'$ is contained in a cell $C''$. Since every vertex in $C' \setminus V(P)$ except $w$ has degree exactly 2, it follows that $|E(C') \cap E(C'')| \geq |E(C')| - |E(P)| = 7$. But this is a contradiction, since $C'$ and $C$ were chosen to be the cycle of length 9 and 7 that have the largest intersection.
\end{proof}


\begin{lemma}\label{9cycles}
Let $C$ and $C'$ be distinct 9-cycles in $G$, with $V(C) \cap V(C') \neq \emptyset$. Their intersection is a path of length at most 2.
\end{lemma}

\begin{proof}
Suppose not, and let $H = C \cup C'$. First suppose the cycles $C$ and $C'$ intersect in at least two paths $P$ and $P'$. Since paths have potential at least 17, we have $p(H) = p(C) + p(C') - p(P) - p(P') \leq 2$, contradicting Lemma \ref{potentials7}. 

Thus we may assume the cycles $C$ and $C'$ intersect in a single path $P$ of length $k \geq 3$. The potential of $H$ is given by $p(H) = p(C) + p(C') - p(P) = 18 + 18 -(2k+17)=19-2k$. Note $H \neq G$ since no theta graph is $C_7$-critical. Thus by Lemma \ref{potentials7}, $p(H) \geq 10$. Hence $k \leq 4$. Note by assumption $k \geq 3$. By Lemma \ref{potentials7}, $G \in P_5(H) \cup P_4(H)$. 

Suppose first $G \in P_5(H)$. Let $Q_5$ be the path with $G = H\cup Q_5$. Let $a$ and $b$ be the endpoints of the path $P = C \cap C'$. Suppose first $(C \setminus V(P)) \cap Q_5 = \emptyset$. Let $a_1a_2 ... a_{9-k-1} = C \setminus V(P)$, and let $b_1 ... b_{9-k-1} = C' \setminus V(P)$, labeled so that $aa_1...a_{9-k-1}bb_{9-k-1}...b_{1}a$ forms a cycle of length $2(9-k)$. Since $G$ is $C_7$-critical, $C' \cup Q_5$ has a homomorphism $\phi$ to $C_7$. But then $\phi$ extends to a homomorphism of $G$ by setting $\phi(a_i) = \phi(b_i)$ for each $i \in \{1, \dots, 9-k-1\}$.

Thus we may assume $ (C \setminus V(P)) \cap Q_5 \neq \emptyset$, and symmetrically, $(C' \setminus V(P)) \cap Q_5 \neq \emptyset$. Let $q \in V(Q_5 \cap C)$. Let $v_1 \neq v_2$ be neighbours of $q$ such that $\{v_1, v_2\} \subset V(C)$. Let $G'$ be the graph obtained from $G$ by identifying $v_1$ and $v_2$ to a new vertex $v$. Note if $G'$ admits a homomorphism $\phi$ to $C_7$, then $\phi$ extends to $G$ by setting $\phi(v_1) = \phi(v_2) = \phi(v)$. Therefore $G'$ contains a $C_7$-critical subgraph $G''$. Note there exists an edge in the $5$-string formed by $Q_5v$ that is not contained in $E(G'')$ by Lemma \ref{max-str}. Since $C_7$-critical graphs have minimum degree two by Lemma \ref{cut}, it follows that $E(Q_5v) \cap E(G'') = \emptyset$. Thus $G''$ is a subgraph of a theta graph $H'$. But no theta graph is $C_7$-critical. Since $H'$ has girth at least 7, we have that $G'' \not \in \{C_3, C_5\}$. But then $H'$ does not contain a $C_7$-critical subgraph, a contradiction.

We may therefore assume $G \not \in P_5(H)$, and so by Lemma \ref{potentials7}, we have $p(H) \geq 12$.  Since $p(H)= 19-2k$, it follows that $k = 3$. By Lemma \ref{potentials7}, $G \in P_4(H)$. Let $Q_4= q_0q_1q_2q_3q_4$ be the path such that $G = H \cup Q_4$. As above, $(C \setminus V(P))\cap Q_4 \neq \emptyset$ as otherwise a homomorphism $\phi: Q_4 \cup C' \rightarrow C_7$ extends to $G$, a contradiction.  Symmetrically, $(C \setminus V(P))\cap Q_4 \neq \emptyset$. Let $q \in V( C \cap Q_4)$, and let $q' \in V( C' \cap Q_4)$. Let $v_1$ and $v_2$ neighbour $q$, with $\{v_1, v_2\} \subset V(C)$. Similarly, let $u_1 \neq u_2$ be neighbours of $q'$ such that $\{u_1, u_2\} \subset V(C')$. Let $G'$ be the graph obtained from $G$ by both identifying $v_1$ and $v_2$ to a new vertex $v$, and identifying $u_1$ and $u_2$ to a new vertex $u$. Note if $G'$ admits a homomorphism $\phi$ to $C_7$, then $\phi$ extends to a homomorphism of $G$ by setting $\phi(v_1) = \phi(v_2) = \phi(v)$, and  $\phi(u_1) = \phi(u_2) = \phi(u)$. Therefore $G'$ contains a $C_7$-critical subgraph $G''$. Note that there exists an edge in the 5-string formed by $uQ_4v$ that is not contained in $G''$ by Lemma \ref{max-str}. Since $C_7$-critical graphs have minimum degree two by Lemma \ref{cut}, it follows $E(uQ_4v) \cap E(G'') = \emptyset$. Thus $G''$ is a subgraph of a theta graph $H'$. Note no theta graph is $C_7$-critical. Since $H'$ has girth at least 7, we have that $G'' \not \in \{C_3, C_5\}$. But then $H'$ does not contain a $C_7$-critical subgraph, a contradiction.
\end{proof}


\subsection{Forbidden Structures}\label{forb}
The lemmas in this section are used to rule out the existence of certain configurations in $G$, and to establish the neighbouring structure of others. Lemmas \ref{(4,4,k)}, \ref{4,3,k}, and \ref{3,3,2} rule out the existence of certain types of vertices of degree three. In Lemma \ref{deg3cells}, we show that $G$ does not contain cells of low degree. Finally, Lemmas \ref{3,2,2s}, \ref{2,2,2s}, and \ref{(3,3,0)} establish the neighbouring structure of certain types of vertices not contained in cells. 

Given the structure established in the previous section, we are now equipped to rule out several types of degree three vertices. We note that in the discharging portion of the proof of Theorem \ref{main}, the problematic structures will be degree three vertices with weight at least six. In ruling out a subset of these types of vertices, we therefore shorten and simplify the discharging portion of the proof of Theorem \ref{main}.


\begin{lemma}\label{(4,4,k)}
$G$ does not contain a vertex of type $(4, 4, k)$, where $0 \leq k \leq 4$. 
\end{lemma}
\begin{proof}
Suppose not. Then there exists a vertex $v$ of type $(4,4,k)$ with neighbours $a, b$, and $c$, where $a$ is contained in a 4-string $S_a$, and $b$ is contained in a 4-string $S_b \neq S_a$ by Lemma \ref{cut}.  Lemma \ref{4str} applied to $v$ and $S_a$ implies the edge $vc$ is contained in a cell $C$. Lemma \ref{4str} applied to $v$ and $S_b$ implies $vc$ is contained in a cell $C' \neq C$. This contradicts Lemma \ref{7cycles}, since distinct cells are vertex-disjoint.
\end{proof}


\begin{lemma}\label{4,3,k}
$G$ does not contain a vertex of type $(4,3,k)$, where $1 \leq k \leq 3$.
\end{lemma}
\begin{proof}
Suppose not. Then for some $k \in \{1,2,3\}$, there exists a vertex $v$ of type $(4,3,k)$. Let $a$ be a vertex in $N(v)$ that is contained in a 4-string $S_a$, and let $b$ be a vertex in $N(v)$ contained in a 3-string $S_b$. Let $c \in N(v) \setminus \{b,a\}$ be the other neighbour of $v$. By applying Lemma \ref{4str} to $v$ and $S_a$, we have that the path $bvc$ is contained in a cell $C$. By applying Lemma \ref{3str} to $v$ and $S_b$, we have that the path $avc$ is contained in a cycle $C'$ of length either seven or nine.  In particular, the edge $vc$ is contained in $E(C' \cap C)$. First suppose $C'$ is of length seven. This contradicts Lemma \ref{7cycles} as distinct cells are vertex-disjoint. Thus we may assume $C'$ is of length nine. This contradicts Lemma \ref{7and9cycles}, since 7-cycles and 9-cycles are edge-disjoint. 
\end{proof}


\begin{lemma}\label{3,3,2}
$G$ does not contain a vertex of type $(3,3,2)$. 
\end{lemma}
\begin{proof}
Suppose not. Then there exists a vertex $v$ of type $(3,3,2)$. Let $a$ be a neighbour of $v$ that is contained in a 3-string $S_a$. Let $b \neq a$ be a neighbour of $v$ contained in a 3-string $S_b$. Finally, let $c$ be the neighbour of $v$ contained in a 2-string $S_c$.  By applying Lemma \ref{3str} to the $v$ and $S_a$, we have that $S_c$ is contained in a cycle $C$ of length either seven or nine. By applying Lemma \ref{3str} to $v$ and $S_b$, we have that the path $S_c$ is contained in a cycle $C' \neq C$ of length either seven or nine. Suppose first $C$ and $C'$ are both cells. Since $S_c \in C \cap C'$, this contradicts Lemma \ref{7cycles} as cycles of length seven are vertex disjoint. Suppose next that one of $C$ and $C'$ is a 9-cycle, and the other is a cell. Since $S_c \in C \cap C'$, this contradicts Lemma \ref{7and9cycles} as 7-cycles and 9-cycles are edge disjoint. Thus we may assume both $C$ and $C'$ are 9-cycles. But this contradicts Lemma \ref{9cycles}, as distinct 9-cycles intersect in a path of length at most two.
\end{proof}

The following lemma is used to lower-bound the degree of cells in $G$. This will be useful in the discharging portion of the proof.  
\begin{lemma}\label{deg3cells}
$G$ does not contain a cell of degree at most two.
\end{lemma}
\begin{proof}
Suppose not. Note since $G$ is $C_7$-critical, $G$ is not a cell. Thus if $G$ contains a cell of degree 0, we have that $G$ is disconnected, contradicting Lemma \ref{cut}. If $G$ contains a cell of degree one, then $G$ contains a cut vertex, contradicting Lemma \ref{cut}. 

We may therefore assume $G$ contains a cell $C$ of degree two. Let $P$ be a longest string contained in $C$. Note $P$ is a $k$-string with $k \geq 3$ since $\deg(C) = 2$. Note that $k\leq 4$ by Lemma \ref{max-str}. Suppose first that $P$ is a 4-string, and let $v$ be an endpoint of $P$. Note since $C$ has degree 2, it follows that $v$ has degree 3. Let $u_1 \neq u_2$ be neighbours of $v$, such that $P \cap \{u_1, u_2\} = \emptyset$. By Lemma \ref{4str}, the path $u_1vu_2$ is contained in a cell $C'$. But since $v$ is also contained in $C$, we have that $V(C) \cap V(C') \neq \emptyset$, contradicting Lemma \ref{7cycles}. 

Thus we may assume $P$ is a 3-string. Let $u \neq v$ be the endpoints of $P$. Let $u_1$ be a vertex in $V(C) \cap N(u)$, with $u_1 \not \in V(P)$. Similarly, let $v_1$ be a vertex in $V(C) \cap N(v)$, with $v_1 \not \in V(P)$.  Let $u_2$ be a vertex in $N(u) \setminus V(C)$, and let $v_2$ be a vertex in $N(v) \setminus V(C)$. Note $v_2 \neq u_2$, as otherwise $v_2Pv_2$ is a cycle of length 6 in $G$, contradicting Lemma \ref{girth}. Furthermore, $v_2$ and $u_2$ are not adjacent as otherwise the cell $v_2Pu_2v_2 \neq C$ intersects $C$ contradicting Lemma \ref{7cycles}. 

Let $G'$ be the graph obtained from $G$ by both identifying $u_1$ and $u_2$ to a new vertex $z_u$, and identifying $v_1$ and $v_2$ to a new vertex $z_v$. 

\noindent{
\textbf{Claim 1.} \emph{$G'$ does not contain a triangle.}}
\begin{proof}
Suppose not. Since $G$ does not contain a 5-cycle, a triangle in $G'$ contains both $z_u$ and $z_v$. But then the path  $u_2uu_1v_1vv_2$ is contained in a cell $C' \neq C$. This contradicts Lemma \ref{7cycles}.
\end{proof}

\noindent{
\textbf{Claim 2.} \emph{$G'$ does not contain a 5-cycle.}
}
\begin{proof}
Suppose not. Let $K$ be a 5-cycle contained in $G'$. Since $K \not \subset G$, we have that at least one of $z_u$ and $z_v$ is contained in $V(K)$. Suppose first exactly one of $z_u$ and $z_v$ is contained in $V(K)$. Without loss of generality, suppose $z_u \in V(K)$. Then the path $u_1uu_2$ is contained in a cell $C'$ in $G$. Since $u \in V(C \cap C')$, this contradicts Lemma \ref{7cycles}. 

Thus we may assume both $z_u$ and $z_v$ are contained in $V(C')$, and that the path $u_2uu_1v_1vv_2$ is contained in a 9-cycle $C'$ in $G$. But then the path $uu_1v_1v$ is contained in both $C'$ and $C$, contradicting Lemma \ref{7and9cycles}.
\end{proof}

Note $G'$ does not admit a homomorphism $\phi$ to $C_7$, as any such homomorphism extends to $G$ by setting $\phi(u_1) = \phi(u_2) = \phi(z_u)$, and $\phi(v_1)=\phi(v_2)=\phi(v)$. Thus $G'$ contains a $C_7$-critical subgraph $G''$. By Claims 1 and 2, we have that $G'' \not \in \{C_3, C_5\}$. Since $v(G'') < v(G)$ and $G$ is a minimum counterexample, $p(G'') \leq 2$.  Note by Lemma \ref{max-str}, $P \not \subseteq G''$ since $z_uPz_v$ is a 5-string. Furthermore, $G''$ contains at least one of $\{z_u, z_v\}$ since $G'' \not \subseteq G$. 

Suppose first $G''$ contains exactly one of $\{z_u, z_v\}$, and without loss of generality suppose $z_v \in V(G'')$. Let $F$ be the graph obtained from $G''$ by splitting $z_v$ back into $v_1$ and $v_2$ and adding $v_1vv_2$. We have $p(F) \leq p(G'') + 17(2)-15(2) \leq 6$, and so $F$ contradicts Lemma \ref{potentials7}.

Thus we may assume $G''$ contains both of $\{z_u, z_v\}$. Let $F$ be the graph obtained from $G''$ by splitting $z_v$ back into $v_1$ and $v_2$ and adding the path $v_1vv_2$. Let $F'$ be obtained from $F$ by splitting $z_u$ back into $u_1$ and $u_2$ and adding the path $u_1uu_2$. We have $p(F') \leq p(G'') + 17(4) - 15(4) \leq 10$. By Lemma \ref{potentials7}, $G \in P_5(F')$ or $F' = G$. But since $P\setminus \{u,v\} \not \in F'$ and $P$ is not a 4-string, this is a contradiction. 
\end{proof}

Finally, the last three lemmas in this section establish the neighbouring structure of certain types of vertices not contained in cells.  These lemmas will be useful in the discharging portion of the proof of Lemma \ref{main}. 

\begin{lemma}\label{3,2,2s}
Let $v \in V(G)$ be a vertex of type (3,2,2) that is not contained in a cell. If $a \neq b$ are the vertices that share a 2-string with $v$ and $\deg(a)=\deg(b)=3$, then at least one of $a$ and $b$ is contained in a cell.
\end{lemma}
\begin{proof}
Suppose not. Let $S_a = aa_1a_2v$ be the 2-string shared by $a$ and $v$, and let $S_b = bb_1b_2v$ be the 2-string shared by $b$ and $v$. Let $S$ be the 3-string incident with $v$. By Lemma \ref{3str} applied to $v$, since $v$ is not contained in a cell there exists a 9-cycle $C = S_a\cup S_b \cup bc_2c_1a$ in $G$. Let $x_1 \in N(a) \setminus V(C)$, and let $x_2 \in N(b) \setminus V(C)$.

Note since $C$ is a 9-cycle, by Lemma \ref{9cycles} the path $x_1ac_1c_2bx_2$ is not contained in a 9-cycle. Furthermore, since cells and 9-cycles are edge-disjoint by Lemma \ref{7and9cycles}, $ac_1$ and $c_2b$ are each not contained in a cell. Let $G'$ be the graph obtained from $G$ by identifying $x_i$ and $c_i$ to a new vertex $z_i$, for each $i \in \{1,2\}$. Note that $G'$ does not contain a triangle or 5-cycle. Furthermore, if $G'$ admits a homomorphism $\phi$ to $C_7$, then $\phi$ extends to $G$ by setting $\phi(x_1) = \phi(c_1) = \phi(z_1)$ and $\phi(x_2) = \phi(c_2) = \phi(z_2)$, contradicting the fact that $G$ is $C_7$-critical. Thus $G'$ does not admit a homomorphism to $C_7$. Furthermore, $G'$ is not $C_7$-critical, as $v$ is a vertex of degree 3 and weight 9 in $G'$, contradicting Lemma \ref{weight}. Thus $G'$ contains a proper $C_7$-critical subgraph $G''$.  Note since $G''$ does not contain a vertex of degree 3 with weight at least 9, at least one edge in one of the strings $S^*$ incident with $v$ is not contained in $G''$. Since $G''$ has minimum degree 2 by Lemma \ref{cut}, we have $E(S^*) \cap E(G'') = \emptyset$. Let $S'$ and $S''$ be the strings in $\{S_a, S_b, S\} \setminus \{S^*\}$. Since $S' \cup S''$ is a $k$-string with $k \geq 5$, at least one of the edges in $S' \cup S''$ is not contained in $E(G'')$. Since $G''$ has minimum degree 2, it follows that $E(S' \cup S'') \cap E(G'') = \emptyset$. In particular, it follows that $v \not \in V(G'')$.

Since $G'' \not \subset G$, it follows that $G''$ contains at least one of $z_1$ and $z_2$. Furthermore, since $G''$ is not a triangle or a 5-cycle and $v(G'') < v(G)$, we have $p(G'') \leq 2$. 

Suppose first exactly one of $\{z_1, z_2\}$ is contained in $V(G'')$. Without loss of generality, we may assume $z_1 \in V(G'')$. Let $F$ be the graph obtained from $G''$ by splitting $z_1$ back into $c_1$ and $x_1$, and adding the path $x_1ac_1$. Then $p(F) = p(G'') + 17(2)-15(2) \leq 6$. Since $F \neq G$, this contradicts Lemma \ref{main}.

Thus we may assume both of $\{z_1, z_2\}$ are contained in $V(G'')$. Let $F$ be the graph obtained from $G''$ by splitting $z_1$ back into $c_1$ and $x_1$, splitting $z_2$ back into $c_2$ and $x_2$, and adding the paths $x_1ac_1$ and $x_2bc_2$. Then $p(F) = p(G'') + 17(4)-15(4) \leq 10$. By Lemma \ref{potentials7}, either $F = G$ or $G \in P_5(H)$. But since $v \not \in V(F)$ is a vertex of degree 3, this is a contradiction.
\end{proof}

\begin{lemma}\label{2,2,2s}
Let $v \in V(G)$ be a vertex of type (2,2,2) not contained in a cell. If $a, b,$ and $c$ are the vertices that share a 2-string with $v$ and $\deg(a)=\deg(b)=\deg(c)=3$, then at least one of $a, b$ and $c$ is contained in a cell.
\end{lemma}
\begin{proof}

Suppose not. Let $S_a$, $S_b$, and $S_c$ be the 2-strings shared by $v$ with $a$, $b$, and $c$, respectively. Let $N(a)\setminus V(S_a) = \{a_1, a_2\}$, let $N(b)\setminus V(S_b) = \{b_1, b_2\}$, and similarly let $N(c)\setminus V(S_c) = \{c_1, c_2\}$. Note first $a$, $b$, and $c$ are all distinct vertices, since $G$ does not contain a 6-cycle by Lemma \ref{girth}. Furthermore, $\{a_1, a_2\} \cap \{b, c\} = \emptyset$ since $v$ is not contained in a cell. Similarly, $\{b_1, b_2\} \cap \{a,c\} = \{c_1, c_2\} \cap \{a, b\} = \emptyset$.  

Let $G'$ be the graph obtained from $G$ by identifying $a_1$ and $a_2$ to a new vertex $z_a$; identifying $b_1$ and $b_2$ to a new vertex $z_b$; and identifying $c_1$ and $c_2$ to a new vertex $z_c$. If $G'$ admits a homomorphism $\phi$ to $C_7$, then $\phi$ extends to $G$ by setting $\phi(x_1) = \phi(x_2) = \phi(z_x)$ for each $x \in \{a,b,c\}$, contradicting the fact that $G$ is $C_7$-critical. Therefore $G'$ contains a $C_7$-critical subgraph $G''$. Note by Lemma \ref{weight}, at least one vertex in $V(S_a) \cup V(S_b) \cup V(S_c)$ is not in $V(G'')$ as otherwise $v$ has weight nine in $G''$, contradicting Lemma \ref{weight}. Without loss of generality, suppose there is a vertex in $V(S_a)$ not contained in $V(G'')$. Since $G''$ is $C_7$-critical, by Lemma \ref{cut} $G''$ has minimum degree 2, and so $(V(S_a) \setminus \{v\}) \cap V(G'') = \emptyset$. Suppose $S_a \cup S_b$ is contained in $G''$. Since $E(S_a) \cap E(G'') = \emptyset$, it follows that $S_a \cup S_b$ is a 7-string in $G''$, contradicting Lemma \ref{max-str}. Thus at least one edge in $E(S_a \cup S_b)$ is not contained in $E(G'')$. Since $G''$ has minimum degree 2 by \ref{cut}, it follows that $E(S_a \cup S_b) \cap E(G'') = \emptyset$, and furthermore that $V(G'') \cap V(S_a \cup S_b \cup S_c) = \emptyset$.

\noindent{
\textbf{Claim 1.} \emph{$G''$ is neither a triangle nor a 5-cycle.}}
\begin{proof}

Suppose not, and suppose first $G''$ is a triangle. Since neither $a$, $b$, nor $c$ is contained in a 5-cycle in $G$ by Lemma \ref{girth}, $G''$ we have that $|\{z_a, z_b, z_c\} \cap V(G'') | \geq 2$. Since neither $a$, $b$, nor $c$ is contained in a cell, it follows that $|\{z_a, z_b, z_c\} \cap V(G'')| = 3$ and that the paths $a_1aa_2$, $b_1bb_2$, and $c_1cc_2$ are contained in a 9-cycle $C$. Let $F = S_a \cup S_b \cup S_c \cup C$. The potential of $F$ is given by
\begin{align*}
p(F) &= p(C) + p(S_a) + p(S_b) + p(S_c) - p(a) - p(b) - p(c) - 2p(v) \\
&=  18+3(23)-3(17)-2(17) \\
&= 2.
\end{align*}  This contradicts Lemma \ref{potentials7}. 

Suppose next $G''$ is a 5-cycle. Since neither $a$, $b$, nor $c$ is contained in a cell, $|\{z_a, z_b, z_c\} \cap V(G'') | \geq 2$. First suppose that $|\{z_a, z_b, z_c\} \cap V(G'') | = 2$, and so that two of the paths $a_1aa_2$, $b_1bb_2$, and $c_1cc_2$ are contained in a 9-cycle $C$. Without loss of generality, assume the paths are $a_1aa_2$ and $b_1bb_2$. Let $F$ be the graph formed by $C \cup S_a \cup S_b$. The potential of $F$ is given by 
\begin{align*}
p(F) &= p(C) + p(S_a) + p(S_b) -p(a) -p(b)-p(v) \\
&= 18+ 2(23) - 3(17) \\
&= 13.
\end{align*}
By Lemma \ref{potentials7}, we have $F = G$ or $G \in P_5(F) \cup P_4(F)$. But this is a contradiction, since $S_c$ is a 2-string and $S_c \not \subset F$.

We may therefore assume that $|\{z_a, z_b, z_c\} \cap V(G'') | = 3$, and so that the paths $a_1aa_2$, $b_1bb_2$ and $c_1cc_2$ are contained in an 11-cycle $C$. Let $F = C \cup S_a \cup S_b \cup S_c$. The potential of $F$ is given by 
\begin{align*}
p(F) &= p(C) + p(S_a) + p(S_b) + p(S_c) - p(a) - p(b) - p(c) - 2p(v) \\
&= 22 + 3(23) - 17(5) \\
&= 6.
\end{align*} 
By Lemma \ref{potentials7}, $F = G$. Let $P_{ab}$ be the $(a,b)$-path in $C$ that does not contain $c$. Similarly, let $P_{bc}$ and $P_{ac}$ be the $(b,c)$- and $(a,c)$-paths along $C$ that do not contain $a$ and $b$, respectively. Note since neither $a$, $b$, nor $c$ is contained in a cell by assumption, neither $P_{ab}$, $P_{bc}$, nor $P_{ac}$ is a 4-string by Lemma \ref{4str}. Since together the three paths form an 11-cycle, we may assume without loss of generality that each of $P_{ab}$ and $P_{bc}$ is a 3-string, and that $P_{ac}$ is a 2-string. Note $S_a \cup S_c \cup P_{ac}$ forms a 9-cycle $C'$. Let $a'$ and $c'$ be $v$'s neighbours in $S_a$ and $S_c$, respectively. Let $F'$ be the graph obtained from $G$ by identifying $a'$ and $c'$. Note $F'$ does not admit a homomorphism to $C_7$ as such a homomorphism extends to $G$. Thus $F'$ contains a $C_7$-critical subgraph $F''$.  Note $C'$ is a cell of weight 9 in $F''$, and so $F''$ does not contain at least one string in or incident with $C'$. But then $F''$ is a subgraph of a theta graph. Note no theta graph is $C_7$-critical; since $F''$ has girth at least 7, it follows that $F''$ is not $C_7$-critical.
\end{proof}

By Claim 1, $G'' \not \in \{C_3, C_5\}$, and since $G$ is a minimum counterexample and $v(G'') < V(G)$, it follows that $p(G'') \leq 2$. Since $G'' \not \subset G$, at least one of $\{z_a, z_b, z_c\}$ is contained in $V(G'')$. Let $I = \{x : z_x \in V(G'')\}$. Let $F$ be the graph obtained from $G''$ by splitting $z_x$ back into $x_1$ and $x_2$ and adding the path $x_1xx_2$, for each $i \in I$. Let $k = |I| \leq 3$. We have $p(F) \leq p(G'') + 17(2k) - 15(2k)$. Note this does not necessarily hold with equality, since $\{x_1, x_2\}$ and $\{y_1, y_2\}$ are not necessarily disjoint for $x \neq y$, $\{x, y\} \subseteq I$. Simplifying, $p(F) \leq p(G'') + 2(2k) \leq 14$. By Lemma \ref{potentials7}, either $F = G$ or $G \in P_5(F) \cup P_4(F) \cup P_3(F)$.  But since $v \not \in V(F)$ and $\deg(v) = 3$, this is a contradiction. 
\end{proof}

\begin{lemma}\label{(3,3,0)}
Let $v \in V(G)$ be a vertex of type (3,3,0) that is not contained in a cell. Let $S$ be the 0-string incident with $v$, and let $u \neq v$ be an endpoint of $S$. If $u$ is not contained in a cell, $u$ is not of type (3,3,0).
\end{lemma}
\begin{proof}
Suppose not. Let $S_a$ and $S_b$ be the 3-strings incident with $v$, and let $S_c$ and $S_d$ be the 3-strings incident with $u$. Let $a \neq v$ and $b \neq v$ be endpoints of $S_a$ and $S_b$, respectively.  Let $c \neq u$ and $d \neq u$ be endpoints of $S_c$ and $S_d$, respectively. Note $a \neq b$ and $c \neq d$ by Lemma \ref{par-str}.  

In this proof, we consider all numerical indices to be taken modulo 7. We aim to show $\{a,b\} \cap \{c,d\} \neq \emptyset$. To see this, suppose not. Let $\phi$ be a homomorphism from $G \setminus (V(S_a \cup S_b \cup S_c \cup S_d) \setminus \{a,b,c,d\})$ to $C=c_1c_2\cdots c_7c_1$. Let $I_1 = B_\phi(v|a, S_a) \cap B_\phi(v|b, S_b)$. Let $I_2 = B_\phi(u |c, S_c) \cap B_\phi(u |d, S_d)$. Note since $G$ is $C_7$-critical, we have $N_{C}(I_1) \cap I_2 = \emptyset$; otherwise, $\phi$ extends to $G$, a contradiction. First, we will show the following:

\noindent{
\textbf{Claim 1.} \emph{Given $\phi$ as described, $\phi(a) \in \{\phi(c), \phi(d)\}$}.}

\begin{proof}
By Lemma \ref{vartheta}, since each of $S_a, S_b, S_c,$ and $S_d$ is a 3-string, each of $|B_\phi(v |a, S_a)|$, $|B_\phi(v |b, S_b)|$,  $|B_\phi(u |c, S_c)|$, and $|B_\phi(u |d, S_d)|$ is at least 5.  Note since $I_1$ and $I_2$ are each the intersection of two subsets of $V(C)$ of size at least 5, it follows that $|I_1| \geq 3$ and $|I_2| \geq 3$. Suppose $|I_1| \geq 4$. Then $|N_{C}(I_1)| \geq 5$ by Lemma \ref{nbrs}, and so $N_{C}(I_1) \cap I_2 \neq \emptyset$. But then $\phi$ extends to $G$, a contradiction. Therefore $I_1$ (and symmetrically, $I_2$) is a set of size 3. 

We may assume without loss of generality that $\phi(a) = c_1$. Then $B_\phi(v |a, S_a) = V(C)\setminus \{c_{7}, c_{2}\}$. Let $\phi(b) = c_j$, with $j \in \{1, \dots, 7\}$. In order to have $|I_2|  = |B_\phi(v | b, S_b) \cap B_\phi(v | a, S_a)| = 3$, we therefore have $\{c_{7}, c_{2}\} \cap \{c_{j-1}, c_{j+1}\} = \emptyset$. Thus $j \in \{4, 5, 2, 7\}$. Suppose $j = 2$. Then $I_1 = \{c_4,c_5,c_6\}$. But then $|N_{C_7}(I_1)| = 5$, and so since $N_{C}(I_1) \cap I_2 \neq \emptyset$, $\phi$ extends to $G$, a contradiction.  The same is symmetrically true if $j = 7$. We may therefore assume $j \in \{4,5\}$.  Without loss of generality, we will take $j = 4$, as the $j = 5$ case corresponds to simply renaming the vertices along the cycle $C$ in the opposite orientation.  Similarly, if $\phi(c) = c_k$, then $\phi(d) \in \{c_{k+3}, c_{k+4}\}$. Note since $\phi(a) = c_1$ and $\phi(b) = c_4$, we have $I_1 = \{c_{1}, c_{4}, c_{6}\}$.  Thus $N_C(I_1) = \{c_7, c_2, c_3, c_5\}$. Since $\phi$ does not extend to $G$, we have that $I_2 = \{c_1, c_4, c_6\}$, and so without loss of generality, $\phi(c) = c_1$ and $\phi(d) = c_4$. Since $\phi(a) = \phi(c)$, this is a contradiction. 
\end{proof}

Let $G'$ be the graph obtained from $G \setminus (V(S_a \cup S_b \cup S_c \cup S_d) \setminus \{a,b,c,d\})$ by adding a 4-string $S_{ac}$ with endpoints $a$ and $c$, and a 4-string $S_{ad}$ with endpoints $a$ and $d$.  Note by assumption $a \neq c$ and $a \neq d$. It follows that $G'$ does not contain a cycle of length three or five, since a cycle containing either $S_{ad}$ or $S_{ac}$ has length at least six. Note if $G'$ admits a homomorphism $\phi$ to $C$, then $\phi$ extends to $G$ by Claim 1, since $\phi(a) \not \in \{\phi(d), \phi(c)\}$.  Thus $G'$ is not homomorphic to $C_7$, and so it contains a $C_7$-critical subgraph $G''$. Since $G''\not \in \{C_3, C_5\}$, since $v(G'') <v(G)$ and since $G$ is a minimum counterexample, it follows that $p(G'') \leq 2$. Note since $G''$ has minimum degree two and $G'' \not \subset G$, we have that $G''$ contains at least one of $S_{ac}$ and $S_{ad}$.  

Suppose first $G''$ contains exactly one of $S_{ac}$ and $S_{ad}$; without loss of generality, assume $S_{ac} \subset G''$. Let $F$ be the graph obtained from $G''$ by deleting $S_{ac}$ and adding $S_auvS_c$. Then $F \subset G$, and $p(F) = p(G'') + 17(4)-15(4) \leq 10$. By Lemma \ref{potentials7}, either $F = G$ or $G \in P_5(F)$. But since $S_d \not \subset F$ and $S_b \not \subset F$, this is a contradiction. 

We may therefore assume $G''$ contains both $S_{ac}$ and $S_{ad}$. Let $F$ be the graph obtained from $G''$ by deleting $S_{ac}$ and $S_{ad}$, and adding in $S_a$, $S_c$, $S_d$, and the edge $uv$. Since this is a net addition of 3 vertices and 3 edges, we have $p(F) = p(G'') + 17(3) - 15(3) \leq 8$. By Lemma \ref{potentials7}, $F = G$. But since $S_b \not \subset F$, this is a contradiction. 

Thus we may assume $a \in \{c, d\}$. Without loss of generality, assume $a=c$. We now break into two cases depending on whether or not $b = d$.

\noindent{
\emph{Case 1: $b=d$.}} Let $G_1 = G\setminus (V(S_a \cup S_b \cup S_c \cup S_d) \setminus \{a, b\})$. Note since $G$ is $C_7$-critical and $G_1 \subset G$, $G_1$ admits a homomorphism $\phi$ to $C_7 = c_1...c_7c_1$, with $\phi(a) = c_1$ and $\phi(b) \in \{c_1, c_2, c_3, c_4\}$. Note $\phi(b) = c_4$, as otherwise $\phi$ extends to $G$, a contradiction. To see this, see Figure \ref{fig:330s}. 

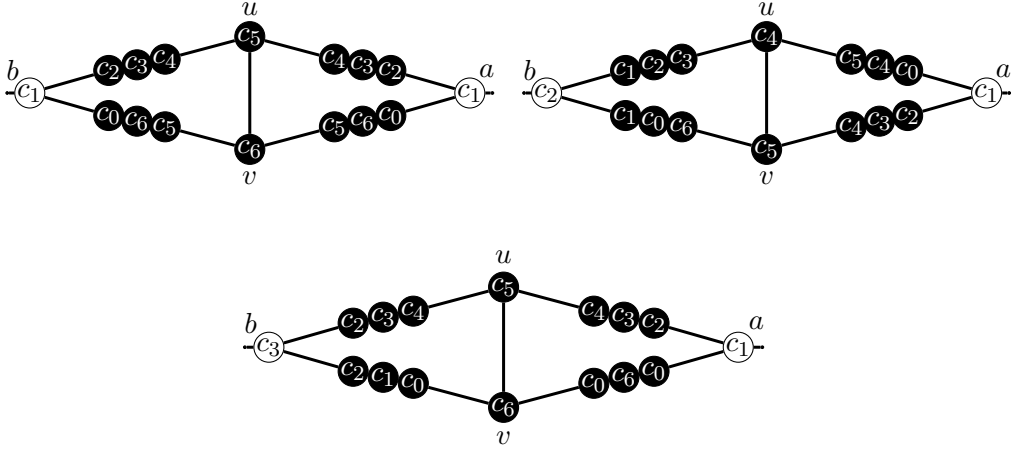
\begin{figure}
	\begin{center}
        
            \begin{tikzpicture}[scale=0.75]	
			\tikzset{black node/.style={shape=circle,draw=black,fill=black,inner sep=0pt, minimum size=11pt}}
            \tikzset{white node/.style={shape=circle,draw=black,fill=white,inner sep=0pt, minimum size=11pt}}
			\tikzset{invisible node/.style={shape=circle,draw=black,fill=black,inner sep=0pt, minimum size=1pt}}
				\tikzset{edge/.style={black, line width=.4mm}}		

                \node[black node] (1) at (0,-1){\color{white}$c_6$};
                \node[black node] (2) at (0,1){\color{white}$c_5$};
                \node[white node] (3) at (3.9,0){$c_1$};                
                \node[white node] (4) at (-3.9,0){$c_1$};                
                \node[black node] (5) at (1.5,0.6){\color{white}$c_4$};                
                \node[black node] (6) at (2,0.5){\color{white}$c_3$};                
                \node[black node] (7) at (2.5,0.4){\color{white}$c_2$};                
                \node[black node] (8) at (1.5,-0.6){\color{white}$c_5$};                
                \node[black node] (9) at (2,-0.5){\color{white}$c_6$};                
                \node[black node] (10) at (2.5,-0.4){\color{white}$c_7$};                
                \node[black node] (11) at (-1.5,0.6){\color{white}$c_4$};                
                \node[black node] (12) at (-2,0.5){\color{white}$c_3$};                
                \node[black node] (13) at (-2.5,0.4){\color{white}$c_2$};                
                \node[black node] (14) at (-1.5,-0.6){\color{white}$c_5$};                
                \node[black node] (15) at (-2,-0.5){\color{white}$c_6$};                
                \node[black node] (16) at (-2.5,-0.4){\color{white}$c_7$};  
                \node[invisible node] (17) at (4.3, 0){};
                \node[invisible node] (18) at (-4.3, 0){};

                \draw[edge] (1)--(2);
                \draw[edge] (1)--(8);
                \draw[edge] (1)--(14);
                \draw[edge] (2)--(5);
                \draw[edge] (2)--(11);
                \draw[edge] (5)--(6);
                \draw[edge] (6)--(7);
                \draw[edge] (7)--(3);
                \draw[edge] (3)--(10);
                \draw[edge] (10)--(9);
                \draw[edge] (9)--(8);
                \draw[edge] (11)--(12);
                \draw[edge] (12)--(13);
                \draw[edge] (13)--(4); 
                \draw[edge] (4)--(16);
                \draw[edge] (16)--(15);
                \draw[edge] (15)--(14); 
                \draw[edge] (3)--(17);
                \draw[edge] (4)--(18);
                \node[] at (0,1.5) {$u$};
                \node[] at (0, -1.5) {$v$};
                \node[] at (4.2,0.4) {$a$};
                \node[] at (-4.2,0.4) {$b$};   
			\end{tikzpicture}
            \begin{tikzpicture}[scale=0.75]	
			\tikzset{black node/.style={shape=circle,draw=black,fill=black,inner sep=0pt, minimum size=11pt}}
            \tikzset{white node/.style={shape=circle,draw=black,fill=white,inner sep=0pt, minimum size=11pt}}
			\tikzset{invisible node/.style={shape=circle,draw=black,fill=black,inner sep=0pt, minimum size=1pt}}
				\tikzset{edge/.style={black, line width=.4mm}}		

                \node[black node] (1) at (0,-1){\color{white}$c_5$};
                \node[black node] (2) at (0,1){\color{white}$c_4$};
                \node[white node] (3) at (3.9,0){$c_1$};                
                \node[white node] (4) at (-3.9,0){$c_2$};                
                \node[black node] (5) at (1.5,0.6){\color{white}$c_5$};                
                \node[black node] (6) at (2,0.5){\color{white}$c_6$};                
                \node[black node] (7) at (2.5,0.4){\color{white}$c_7$};                
                \node[black node] (8) at (1.5,-0.6){\color{white}$c_4$};                
                \node[black node] (9) at (2,-0.5){\color{white}$c_3$};                
                \node[black node] (10) at (2.5,-0.4){\color{white}$c_2$};                
                \node[black node] (11) at (-1.5,0.6){\color{white}$c_3$};                
                \node[black node] (12) at (-2,0.5){\color{white}$c_2$};                
                \node[black node] (13) at (-2.5,0.4){\color{white}$c_1$};                
                \node[black node] (14) at (-1.5,-0.6){\color{white}$c_6$};                
                \node[black node] (15) at (-2,-0.5){\color{white}$c_7$};                
                \node[black node] (16) at (-2.5,-0.4){\color{white}$c_1$};  
                \node[invisible node] (17) at (4.3, 0){};
                \node[invisible node] (18) at (-4.3, 0){};

                \draw[edge] (1)--(2);
                \draw[edge] (1)--(8);
                \draw[edge] (1)--(14);
                \draw[edge] (2)--(5);
                \draw[edge] (2)--(11);
                \draw[edge] (5)--(6);
                \draw[edge] (6)--(7);
                \draw[edge] (7)--(3);
                \draw[edge] (3)--(10);
                \draw[edge] (10)--(9);
                \draw[edge] (9)--(8);
                \draw[edge] (11)--(12);
                \draw[edge] (12)--(13);
                \draw[edge] (13)--(4); 
                \draw[edge] (4)--(16);
                \draw[edge] (16)--(15);
                \draw[edge] (15)--(14); 
                \draw[edge] (3)--(17);
                \draw[edge] (4)--(18);
                \node[] at (0,1.5) {$u$};
                \node[] at (0, -1.5) {$v$};
                \node[] at (4.2,0.4) {$a$};
                \node[] at (-4.2,0.4) {$b$};   
			\end{tikzpicture}
            \vskip 6mm
                        \begin{tikzpicture}[scale=0.8]	
			\tikzset{black node/.style={shape=circle,draw=black,fill=black,inner sep=0pt, minimum size=11pt}}
            \tikzset{white node/.style={shape=circle,draw=black,fill=white,inner sep=0pt, minimum size=11pt}}
			\tikzset{invisible node/.style={shape=circle,draw=black,fill=black,inner sep=0pt, minimum size=1pt}}
				\tikzset{edge/.style={black, line width=.4mm}}		

                \node[black node] (1) at (0,-1){\color{white}$c_6$};
                \node[black node] (2) at (0,1){\color{white}$c_5$};
                \node[white node] (3) at (3.9,0){$c_1$};                
                \node[white node] (4) at (-3.9,0){$c_3$};                
                \node[black node] (5) at (1.5,0.6){\color{white}$c_4$};                
                \node[black node] (6) at (2,0.5){\color{white}$c_3$};                
                \node[black node] (7) at (2.5,0.4){\color{white}$c_2$};                
                \node[black node] (8) at (1.5,-0.6){\color{white}$c_7$};                
                \node[black node] (9) at (2,-0.5){\color{white}$c_6$};                
                \node[black node] (10) at (2.5,-0.4){\color{white}$c_7$};                
                \node[black node] (11) at (-1.5,0.6){\color{white}$c_4$};                
                \node[black node] (12) at (-2,0.5){\color{white}$c_3$};                
                \node[black node] (13) at (-2.5,0.4){\color{white}$c_2$};                
                \node[black node] (14) at (-1.5,-0.6){\color{white}$c_7$};                
                \node[black node] (15) at (-2,-0.5){\color{white}$c_1$};                
                \node[black node] (16) at (-2.5,-0.4){\color{white}$c_2$};  
                \node[invisible node] (17) at (4.3, 0){};
                \node[invisible node] (18) at (-4.3, 0){};

                \draw[edge] (1)--(2);
                \draw[edge] (1)--(8);
                \draw[edge] (1)--(14);
                \draw[edge] (2)--(5);
                \draw[edge] (2)--(11);
                \draw[edge] (5)--(6);
                \draw[edge] (6)--(7);
                \draw[edge] (7)--(3);
                \draw[edge] (3)--(10);
                \draw[edge] (10)--(9);
                \draw[edge] (9)--(8);
                \draw[edge] (11)--(12);
                \draw[edge] (12)--(13);
                \draw[edge] (13)--(4); 
                \draw[edge] (4)--(16);
                \draw[edge] (16)--(15);
                \draw[edge] (15)--(14); 
                \draw[edge] (3)--(17);
                \draw[edge] (4)--(18);
                \node[] at (0,1.5) {$u$};
                \node[] at (0, -1.5) {$v$};
                \node[] at (4.2,0.4) {$a$};
                \node[] at (-4.2,0.4) {$b$};   
			\end{tikzpicture}
            \hskip 8mm

    \end{center}
	\caption{Figure for Lemma \ref{(3,3,0)}. Extensions of $\phi$ to $G$. The white vertices are of unknown degree, though their degree is at least that shown. The black vertices' degrees are as illustrated.}
	\label{fig:330s}
\end{figure}

Let $G_2 \in P_2(G_1)$ be the graph obtained from $G_1$ by adding an $(a,b)$-path $P$ of length 2. Note if $G_2$ admits a homomorphism $\phi$ to $C_7$, then $\phi$ extends to $G$ since there does not exist a homomorphism $\phi': G_2 \rightarrow C_7$ with $\phi'(a) = c_1$ and $\phi'(b) = c_4$. Thus $G_2$ contains a $C_7$-critical subgraph $G_2'$. Suppose $G_2'$ is a triangle. Then $ab$ is an edge in $G$. But then $C = S_a \cup S_c \cup uv$ and $C' = S_b \cup S_a \cup ab$ are two 9-cycles that intersect in a 3-string $S_a$, contradicting Lemma \ref{potentials7}. Suppose next that $G_2'$ is a 5-cycle. Then there exists an $(a,b)$-path $Q$ of length 3 in $G$. Let $F = S_a\cup S_b \cup S_c \cup S_d \cup uv \cup Q$. Since $v(F) = 18$ and $e(F) = 20$, the potential of $F$ is given by $p(F) = 17(18) - 15(20) = 6$. By Lemma \ref{potentials7}, $F = G$. But since there exists a homomorphism $\phi$ of $Q$ to $C_7$ with $\phi(a) = c_1$ and $\phi(b) = c_2$, we have that $\phi$ extends to a homomorphism of $G$ to $C_7$, contradicting the fact that $G$ is $C_7$-critical.  Thus we may assume $G_2' \not \in \{C_3, C_5\}$, and since $v(G_2') < v(G)$ and $G$ is a minimum counterexample, it follows that $p(G_2') \leq 2$.  Note since $G_2' \not \subset G$ and $G_2'$ has minimum degree at least 2 by Lemma \ref{cut}, the path $P$ is contained in $G_2'$. Let $F$ be the graph obtained from $G_2'$ by deleting $V(P) \setminus \{a,b\}$ and adding $S_a \cup S_b$.  The potential of $F$ is given by $p(F) = p(G_2') + 17(6)-15(6) \leq 14$. By Lemma \ref{potentials7}, either $F = G$ or $G \in P_5(F) \cup P_4(F) \cup P_3(F)$. But since one of $u$ and $v$ is not contained in $F$ and $\deg(u) = \deg(v) = 3$, this is a contradiction. 

\noindent{\emph{Case 2: $b \neq d$.}} Let $G'$ be the graph obtained from $G \setminus (V(S_a \cup S_b \cup S_c \cup S_d) \setminus \{a,b,d\})$ by adding a 4-string $S_{bd}$ with endpoints $b$ and $d$.  Note since $b \neq d$, $G'$ does not contain a cycle of length 3 or 5, since a cycle containing $S_{bd}$ has length at least 6. 

Suppose $G'$ admits a homomorphism $\phi$ to $C_7$. Without loss of generality, we may assume $\phi(a) = c_1$ and $\phi(b) \in \{c_1, c_2, c_3, c_4\}$. Note $\phi(b) \neq \phi(d)$, since $B_\phi(b |d, S_{bd}) = V(C_7)- \phi(b)$. 

\noindent{\textbf{Claim 2.}\emph{ The homomorphism $\phi$ extends to $G$.}}

\begin{proof}
 Let $I_1 = B_\phi(v|a, S_a) \cap B_\phi(v|b, S_b)$ and $I_2 = B_\phi(u|a, S_c) \cap B_\phi(u|d, S_d)$. Note since $G$ is $C_7$ critical, it follows that $I_1 \cap N_{C_7}(I_2) = \emptyset$ as otherwise $\phi$ extends to $G$. Since each of $S_a, S_b, S_c,$ and $S_d$ is a 3-string, by Lemma \ref{vartheta} each of $|B_\phi(v |a, S_a)|$, $|B_\phi(v |b, S_b)|$,  $|B_\phi(u |a, S_c)|$, and $|B_\phi(u |d, S_d)|$ is at least 5.  Note since $I_1$ and $I_2$ are each the intersection of two subsets of $V(C_7)$ of size at least 5, it follows that $|I_1| \geq 3$ and $|I_2| \geq 3$. Suppose $|I_1| \geq 4$. Then $|N_{C_7}(I_1)| \geq 5$ by Lemma \ref{nbrs}, and so $N_{C_7}(I_1) \cap I_2 \neq \emptyset$. But then $\phi$ extends to $G$, a contradiction. Therefore $I_1$ (and symmetrically, $I_2$) is a set of size exactly 3. 

Since $\phi(a) = c_1$, it follows that $B_\phi(v | a, S_a) = \{c_1,c_3,c_4,c_5,c_6\}$. Since $\phi(b) \in \{c_1,c_2,c_3,c_4\}$ and $B_\phi(v|b, S_b) = \{c_i: i \in [7], i \neq \phi(b)\pm 1\}$, it follows that $\phi(b) \in \{c_2, c_4\}$. 

Suppose $\phi(b) =c_2$. Then we have
\begin{align*}
I_1 &= B_\phi(v|a, S_a) \cap B_\phi(v|b, S_b) \\
&= \{c_1,c_3,c_4,c_5,c_6\} \cap \{c_2, c_4, c_5, c_6, c_7\} \\
&= \{c_4, c_5, c_6\}.
\end{align*}
But then $|N_{C_7}(I_1)| = 5$, and so it follows that $I_2 \cap N_{C_7}(I_1) \neq \emptyset$ since $I_2$ is a set of size 3. This contradicts the fact that $\phi$ does not extend to $G$. Thus we may assume $\phi(b) = c_4$.

Similarly, we may assume $\phi(d) \neq c_2$, and symmetrically, $\phi(d) \neq c_7$. Thus $\phi(d) \in \{c_4, c_5\}$. Since $\phi(d) \neq \phi(b)$, we have that $\phi(d) = c_5$. But then $I_1 = \{c_1,c_4,c_6\}$ and so $N(I_1) = \{c_2,c_3,c_5,c_7\}$. Since $I_2 = \{c_1,c_3,c_5\}$, we have that $N(I_1) \cap I_2 \neq \emptyset$, and so that $\phi$ extends to $G$.
\end{proof}

Since $G$ is $C_7$-critical, Claim 2 is a contradiction and so we may assume that $G'$ does not admit a homomorphism $\phi$ to $C_7$. Thus $G'$ contains a $C_7$-critical subgraph $G''$. Note since $G'$ does not contain a cycle of length 3 or 5, it follows that $G'' \not \in \{C_3, C_5\}$. Furthermore, since $v(G'') < v(G)$ and $G$ is a minimum counterexample, we have that $p(G'') \leq 2$. Note at least one edge in $E(S_{bd})$ is contained in   $E(G'')$ since $G'' \not \subset G$. Furthermore, since $G''$ has minimum degree 2 by Lemma \ref{cut}, it follows that $S_{bd} \subset G''$. Let $F$ be the graph obtained from $G''$ by deleting $S_{bd} \setminus \{b,d\}$ and adding $S_buvS_d$. Since this is a net addition of 4 vertices and 4 edges, it follows that $p(F) = p(G'') + 17(4)-15(4) \leq 10$. By Lemma \ref{potentials7}, it follows that either $F = G$ or $G \in P_5(F)$. But this is a contradiction, since $S_a \not \subset F$ and $S_c \not \subset F$.

\end{proof}

\section{Discharging}
Now that we have established the required structure of $G$, this section will be dedicated to proving Theorem \ref{main} via discharging. The discharging will proceed in five stages: in each stage, charge will only be sent to structures that have not received charge in previous stages. It follows, then, that a structure in need of charge will only ever receive charge in a single stage.  

Let $X \subseteq V(G)$ be the set of vertices of degree at least three. We assign an initial charge of $ch_0(v) = 15\deg(v)-2\textnormal{wt}(v)-34$ to each vertex $v \in X$, and $ch_0(v) = 0$ for each $v \in V(G) \setminus X$. Note $\sum_{v \in X} (15\deg(v)-2\textnormal{wt}(v)-34) = \sum_{v \in X} (15\deg(v)-34) - \sum_{v \in V(G) \setminus X} 4$, since every vertex $v$ of degree 2 contributes to the weight of two distinct vertices in $X$ (namely, the endpoints of the string that contain $v$). Since $\sum_{v \in V(G) \setminus X} 4 = \sum_{v \in V(G) \setminus X} (34-15 \deg(v))$, we have 

\begin{align}
\sum_{v \in V(G)} ch_0(v) &= \sum_{v \in V(G)} (15\deg(v)-34) \nonumber \\ 
& = 15\sum_{v \in V(G)} \deg(v) - \sum_{v \in V(G)}  34  \nonumber \\
&= 30 e(G)-34 v(G) \nonumber \\
&= -2 p(G) \nonumber\\
&\leq -6 \textnormal{, since $p(G) \geq 3$ and potential is integral.} \label{threshhold}
\end{align}  We will redistribute the charge amongst the vertices and cells until every vertex and cell has non-negative charge, contradicting the fact that the sum of the charges is at most $-6$. The proof will be done in two sections: in Section \ref{nonewnegs}, we will show that after discharging no structures that start with non-negative charge end with negative charge. In Section \ref{getpaid}, we will show that all structures that start with negative charge end with non-negative charge.

For simplicity, we define the following term.
\begin{defs}
A vertex is $poor$ if it has negative charge.
\end{defs}

Note by Lemma \ref{weight}, if $v$ is a vertex in $V(G)$, then $\textrm{wt}(v) \leq 5\deg(v) - 7$. For a vertex $v \in X$, we therefore have $ch_0(v) \geq 15\deg(v)-2(5\deg(v)-7)-34 = 5\deg(v)-20$.  Therefore the only possibly poor vertices are vertices of degree three. If $v$ has degree three and is poor, then it has weight at least six since $ch_0(v) = 11 - 2\textrm{wt}(v)$. By Lemma \ref{weight}, vertices of degree three (and thus poor vertices) have weight at most eight. By Lemmas \ref{(4,4,k)}, \ref{4,3,k}, and \ref{3,3,2} the only poor vertices of weight eight are of type (4,2,2). The poor vertices of weight seven are of type (4,3,0), (4,2,1), (3,3,1), or (3,2,2), and the poor vertices of weight six are of type (4,2,0), (4,1,1), (3,3,0), (3,2,1), or (2,2,2). 

We will discharge in steps: for $i \in \{1,2,3,4\}$, Step $i$ consists of a single rule $Ri$ that will be carried out instantaneously throughout the graph. Step 5 consists of iteratively applying Rule 5 $x$ times, where $x$ is the number of vertices of degree three and weight five in $G$. Each application of R5 will be carried out instantaneously throughout the graph.  For convenience, since a single rule is carried out in each step, we will refer to the rules and steps interchangeably. At the end of Step $i$, the resulting charge of each cell and vertex will be denoted by $ch_i$. Finally, a $k$-string in $G$ will be called \emph{short} if $k \leq 2$.

\begin{itemize}
\item [\textbf{R1.}] Each vertex contained in a cell sends all of its charge to the cell that contains it. (Since cells are disjoint by Lemma \ref{7cycles}, this is unambiguous.) 
\item [\textbf{R2.}] Let $u$ and $v$ share a short string. If $u$ is in a cell $C$ and $v$ is poor after Step 1, $C$ sends $-ch_1(v)$ charge to $v$.
\item [\textbf{R3.}] Let $u$ and $v$ share a short string with $\deg(u) \geq 4$. If $v$ is poor after Step 2, $u$ sends  $-ch_2(v)$ charge to $v$.
\item [\textbf{R4.}] Let $u$ and $v$ share a short string with $\deg(u) = 3$ and $\textrm{wt}(u) \leq 4$. If $v$ is poor after Step 3, $u$ sends $-ch_3(v)$ charge to $v$.
\item [\textbf{R5.}] Let $u$ and $v$ share a short string with $\deg(u) = 3$ and $\textrm{wt}(u) = 5$. If $v$ is the only poor vertex that shares a short string with $u$, then $u$ sends $-ch_4(v)$ charge to $v$. (Note in this step, each vertex sends charge to at most one other vertex, though each vertex might receive charge from several others.)
\end{itemize}

Before proceeding with the proof, we note two important facts regarding the discharging rules. First, the rules are performed sequentially. This ensures that in the later steps of the discharging process, we will have uncovered a significant amount of information regarding the local structure of the vertices and cells receiving charge. Second, vertices and cells only send charge along short strings. If a vertex or cell sends charge to many poor structures, it follows that the vertex or cell sending charge has relatively low weight and so consequently has a large amount of charge to spare. 
\subsection{No New Negative Structures are Created}\label{nonewnegs}
In this section, we will show that no cell or vertex $x$ with initial charge $ch_0(x) \geq 0$ has negative final charge after discharging. First we will show that all cells have non-negative charge at the end of the discharging process (Lemma \ref{Step1-2}). We will then prove that no vertex with degree at least four is poor after Step 5 (Lemma \ref{Step3}). Finally with Lemmas \ref{Step4} and \ref{Step5}, we will prove that all vertices of degree 3 and weight at most 5 have non-negative final charge. In Section \ref{getpaid}, we will prove that all vertices of degree 3 and weight at least 6 have non-negative final charge. As cells and vertices are the only structures that carry charge at any point during the discharging, this will show that the sum of the charges is non-negative, contradicting our initial assumption and completing the proof of Theorem \ref{main}.

Throughout the analysis, we will use the following observation repeatedly.

\begin{obs}\label{atleast-3}
If $v$ is a vertex that is poor after Step 1, then $ch_1(v) \geq -3$. If $v$ is incident with a 4-string, $v$ is not poor after Step 1.
\end{obs}
This follows from the fact that all vertices $v$ with $ch_0(v) = -5$  are incident with a 4-string, and by Lemma \ref{4str} all vertices incident with 4-strings are contained in cells (and so have $ch_1(v) = 0$).
\begin{lemma}\label{Step1-2}
Let $C$ be a cell in $G$. At the end of Step 2, $ch_2(C) \geq 0$.
\end{lemma}
\begin{proof}
Let $X$ be the set of vertices in $C$ of degree at least three. At the end of Step 1, $ch_1(C) = \sum_{v \in X}(15\deg(v)-2\textrm{wt}(v)-34)$. Rewriting, 

\begin{align*}
ch_1(C) &= \sum_{v \in X}((15\deg(v)-30)-(2\textrm{wt}(v)+4)) \\
&=  15 \sum_{v \in X}(\deg(v)-2)-2 \sum_{v \in X} (\textrm{wt}(C)+2).
\end{align*} 

But since $\sum_{v \in X} (\textrm{wt}(v) + 2) = \textrm{wt}(C) + 14$ and $\sum_{v \in X} (\deg(v)-2) = \deg(C)$, we have $ch_1(C) = 15 \deg(C)-28-2 \textrm{wt}(C)$.  We will split our further analysis into two cases depending on the degree of the cell $C$. Note $\deg(C) \geq 3$ by Lemma \ref{stringsincells}.\\

\noindent{
\textbf{Case 1:} $\deg(C) \geq 4$.} Suppose that immediately after Step 1, there are $p$ poor vertices that are each the endpoint of a short string whose other endpoint is in $C$. Note each of these $p$ strings contributes at most 2 to the weight of $C$. For each poor vertex $u$ we have $ch_1(u) \geq -3$ by Observation \ref{atleast-3}, and so $ch_2(C) \geq ch_1(C) - 3p = 15\deg(C)-28-2\textrm{wt}(C) - 3p$. Since $\textrm{wt}(C) \leq 4(\deg(C)-p) + 2p$, we have 

\begin{align*}
ch_2(C) &\geq 15\deg(C)-28-2(4(\deg(C)-p)+2p)-3p \\
&= 7\deg(C)-28+p.
\end{align*} Since $\deg(C) \geq 4$, $ch_2(C) \geq p \geq 0$, as desired. \\

\noindent{
\textbf{Case 2:} $\deg(C) = 3$.} Suppose for a contradiction that $ch_2(C) < 0$. 

Suppose first $|X| = 2$. Then $X$ contains a single vertex of degree 3, and a single vertex of degree 4. Note all vertices that are poor immediately after Step 1 have weight at most seven, by Observation \ref{atleast-3}. Let $v$ be the vertex of degree 3 in $X$. Let the three strings incident with $v$ be $S_1$, $S_2$ and $S_3$, named such that $S_1$ and $S_2$ are contained in $C$. Suppose that $S_1$ is a 4-string. By Lemma \ref{4str}, $S_2 \cup S_3$ is contained in a cell $C' \neq C$, contradicting Lemma \ref{7cycles}. Thus we may assume $S_1$ is not a 4-string. Symmetrically, $S_2$ is not a 4-string. Since $|X| = 2$ by assumption, we may therefore assume without loss of generality that $S_1$ is a 3-string and $S_2$ is a 2-string. Note $S_3$ is therefore a 0-string. (To see this, suppose not. Then by Lemma \ref{3str} applied to $v$ and $S_1$, we have that $S_2 \cup S_3$ (and in particular, $S_2$) is contained in either a cell or a 9-cycle $C' \neq C$. We claim this cannot be: to see this, it suffices to note that $S_2 \in C \cap C'$, but by Lemma \ref{7cycles} distinct cells are vertex-disjoint and by Lemma \ref{7and9cycles}, cells and 9-cycles are edge-disjoint.) Let $v_1$ be the vertex that shares $S_3$ with $v$. Note if $ch_1(v_1) <0$, then $v_1$ is not incident with a 4-string by Observation \ref{atleast-3}. Since $v_1$ is incident with a 0-string $S_3$, we have that $\textrm{wt}(v_1) \leq 6$. Thus if $ch_1(v_1) < 0$, it follows that $ch_1(v_1) = -1$. Let $u$ be the vertex of degree 4 in $X$, and let $S_4 \not \subseteq C$ and $S_5 \not \subseteq C$ be the strings incident with $u$, with endpoints $u_1 \neq u$ and $u_2 \neq u$, respectively. Note since $ch_1(C) = 17-2\textrm{wt}(C)$, if $\textrm{wt}(C) \leq 5$, then $ch_2(C) \geq 0$ since $C$ pays at most 1 to $v_1$ in Step 2, and at most 3 to each of $u_1$ and $u_2$ in Step 2. Thus we may assume $\textrm{wt}(C) \geq 6$. Note also if neither $u_1$ nor $u_2$ is poor immediately after Step 1, then $ch_2(C) \geq 0$ since $ch_2(C) \geq 17-2\textrm{wt}(C) -ch_1(v_1)$, and $\textrm{wt}(C) \leq 8$ by Lemma \ref{cell-weight}. Thus we may assume at least one of $u_1$ and $u_2$ receives charge from $C$ in Step 2, and so that at least one of $S_4$ and $S_5$ is a short string. Note that $G$ does not contain a $k$-string with $k \geq 5$ by Lemma \ref{max-str}. Furthermore, since at least one of $S_4$ and $S_5$ is short, $S_4$ and $S_5$ together contribute at most 6 to the weight of $C$. Thus $\textrm{wt}(v) \leq 6$, and since $\textrm{wt}(C) \geq 6$, it follows that $\textrm{wt}(C) = 6$. Since $S_3$ is a 0-string and at least one of $S_4$ and $S_5$ is a short string, it follows that exactly one of $S_4$ and $S_5$ is a short string, and so that $C$ sends charge to exactly one of $u_1$ and $u_2$. But then $ch_2(C)\geq 0$, since $ch_1(C) = 17-2 \textrm{wt}(C) = 5$ and $C$ sends at most 1 to $v_1$ and 3 to one of $u_1$ and $u_2$.

Thus we may assume $|X| = 3$. Since $\deg(C) = 3$, it follows that $X$ contains three vertices of degree three. Let $S_0$ be a $k$-string with $k \leq 2$ and endpoints $u,v$ such that  $v \in V(C)$. Suppose that $u$ is poor after Step 1. Then $u$ is a degree three vertex not contained in a cell. Let $S_1$ and $S_2$ be the other strings incident with $u$. \\

\noindent{
\textbf{Claim 1.} \emph{The weight of $u$ is exactly six.}
} 
\begin{proof}
Suppose not. Note as $u$ is poor after Step 1, it has weight at least six. Note neither $S_1$ nor $S_2$ is a 4-string by Observation \ref{atleast-3}. Suppose now $k \in \{1, 2\}$ (i.e. that $S_0$ is either a 1- or 2-string). First suppose that $S_1$ is a 3-string and that $S_2$ is not a 0-string. Then by Lemma \ref{3str} applied to $u$ and $S_1$, there exists an edge $vv_1 \in E(C)$ contained in a cell or 9-cycle $C' \neq C$. If $C'$ is a 9-cycle, this is a contradiction since 9-cycles and cells are edge-disjoint by Lemma \ref{7and9cycles}. If $C'$ is a cell, this too is a contradiction since distinct cells are vertex-disjoint by Lemma \ref{7cycles}. Thus if $S_1$ is a 3-string, we may assume $S_2$ is a 0-string. Symmetrically, if $S_2$ is a 3-string, we may assume $S_1$ is a 0-string.

If $k \geq 1$, then $S_1$ and $S_2$ therefore contribute at most 4 to the weight of $C$. If $k = 0$, since neither $S_1$ nor $S_2$ is a 4-string, they contribute at most 6 to the weight of $C$.  Note by assumption $k \leq 2$. Thus $u$ has weight at most six. Since by assumption $ch_2(u) < 0$, it follows that $u$ has weight exactly six. 
\end{proof}

As $\textrm{wt}(u) = 6$, we have that $ch_2(u) = -1$. Let $A$ be the set of vertices that are poor immediately after Step 1, and that are incident with a short string incident with $C$.  Let $|A| = p$. Since $C$ has degree 3, it follows that $p \leq 3$. After Step 2, we have $ch_2(C) = ch_1(C) - p = 17 - 2\textrm{wt(C)} - p$. Note if $\textrm{wt}(C) \leq 7$, then since $p \leq 3$ we have that $ch_2(C) \geq 0$, a contradiction. Thus we may assume that $\textrm{wt}(C) \geq 8$. Note also that by Lemma \ref{cell-weight}, $\textrm{wt}(C) \leq 8$, and so we may assume $\textrm{wt}(C) = 8$. Since $ch_2(C) < 0$, we have $p \geq 2$. Thus at least two of the strings incident with $C$ each contribute at most 2 to the weight of $C$. Since $C$ has weight 8, it follows that $C$ is incident with a 4-string and so that $p \leq 2$. Thus we may assume $p=2$, and since $\textrm{wt}(C) = 8$, it follows that $C$ is a cell of type $(4,2,2)$.  

Let $S = u_0u_1u_2u_3$ be a 2-string incident with $C$, such that $u_0 \in V(C)$. Note since $u_3$ is poor after Step 2, it is not contained in a cell. Let $u_4$ and $u_5$ be the neighbours of $u_3$ that are not contained in $S$.

Let $G'$ be the graph obtained from $G$ by identifying $u_4$ and $u_5$ to a new vertex $z$. Note $G'$ contains a cell of weight nine, and so $G'$ is not $C_7$-critical by Lemma \ref{cell-weight}.  Furthermore, $G'$ admits no homomorphism to $C_7$, as any such homomorphism $\phi$ extends to $G$ by setting $\phi(u_4) = \phi(u_5) = \phi(z)$. Therefore $G'$ contains a $C_7$-critical subgraph $G''$, and since $G'' \not \subset G$, we have $z \in V(G'')$.  $G'$ contains no 5-cycles nor triangles, since $u_3 \in V(G)$ is not contained in a 7-cycle by assumption. Since $v(G'') \leq v(G)$, $G''$ is not a counterexample to Theorem \ref{main} and thus $p(G'') \leq 2$. Note at least one string incident with $C$ or at least one string $S' \subset C$ is not contained in $G''$, as otherwise $C$ has weight $9$ in $G''$, contradicting Lemma \ref{cell-weight}.
 
 Suppose first $u_3 \not \in V(G'')$. Since $G''$ has minimum degree at least 2 by Lemma \ref{cut}, it follows that $u_2 \not \in V(G'')$. Let $F$ be the graph obtained from $G''$ by splitting $z$ back to $u_4, u_5$ and adding in the path $u_4u_3u_5$. Then $p(F) = p(G'') + 17(2) - 15(2) \leq 6$. Since $u_2 \not \in V(G'')$ and since $F \subset G$, this contradicts Lemma \ref{potentials7}. 
 
 We may therefore assume that $u_3 \in V(G'')$, and so there exists a string $S' \neq S$ whose internal vertices are not contained in $V(G'')$. Note either $S'$ is incident with $C$, in which case it has at least 2 internal vertices, or $S' \subset C$, in which case $V(C)\cap V(G'') = \emptyset$ since $G''$ contains no $k$-string with $k\geq 5$ by Lemma \ref{max-str}. Thus we have $V(G) \setminus V(G'') \neq \emptyset$. Let $F$ be the graph obtained from $G''$ by splitting $z$ back to $u_4$ and $u_5$, and adding edges to create the path $u_4u_3u_5$. Then $p(G) = p(G'') + 17(1)-15(1) \leq 6$. Since $F \subset G$ but $F \subsetneq G$, again this contradicts Lemma \ref{potentials7}.
\end{proof}

\begin{lemma}\label{Step3}
Let $v \in V(G)$ be a vertex of degree at least four. At the end of Step 3, $ch_3(v) \geq 0$.
\end{lemma}
\begin{proof}
Suppose not. Let $A$ be the set of vertices that are poor immediately after Step 2 and that each share a short string with $v$. Let $p = |A|$. Note none of the $p$ vertices in $A$ are contained in cells, as they are poor after Step 2; furthermore, $v$ is not contained in a cell, as this cell would send charge to the vertices in $A$ in Step 2.  By Observation \ref{atleast-3},  $v$ sends at most $3p$ units of charge to the vertices of $A$ in Step 3.

First, suppose $\deg(v) \geq 5$. At the end of Step 3, we have $ch_3(v) \geq ch_2(v) - 3p =15\deg(v)-2\textrm{wt}(v)-34-3p$. Note since $v$ is incident with $p$ short strings and the remaining $\deg(v)-p$ strings incident with $v$ each contribute at most 4 to the weight of $v$, it follows that $\textrm{wt}(v) \leq 4(\deg(v)-p)+2p$. Thus $ch_3(v) \geq  15\deg(v)-2(4\deg(v)-2p)-34-3p = 7\deg(v)-34+p \geq 1+p$, since $\deg(v) \geq 5$. Since $p$ is non-negative, $ch_3(v) \geq 0$, a contradiction. 

Thus we may assume $\deg(v) = 4$. Since $v$ is not contained in a cell, by Lemma \ref{deg4s} $v$ is not incident with a 4-string. Since $v$ is incident with at least $p$ short strings and at most ($\deg(v)-p$) $k$-strings with $k=3$, it follows that $v$ has weight at most $3(\deg(v)-p)+2p = 12-p$, and so 
\begin{align*}
ch_3(v) \geq ch_2(v) - 3p &= 15\deg(v)-2\textrm{wt}(v)-34 - 3p \\
&\geq 60-2(12-p)-34-3p \\
&= 2-p.
\end{align*} 
Thus if $p \leq 2$, we have $ch_3(v) \geq 0$, a contradiction.

Since $p \leq \deg(v) = 4$, we may therefore assume $p \in \{3,4\}$. Note if $\textrm{wt}(v) \leq 7$, since 
\begin{align*}
ch_3(v) &\geq 15\deg(v)-2\textrm{wt}(v)-34 - 3p \\
&\geq 60-14-34-3p \\
&= 12-3p
\end{align*} it follows that $ch_3(v) \geq 0$, a contradiction. 

Thus $v$ has weight at least $8$ and is incident with at least three short strings. It follows that $v$ is either a vertex of type $(3,2,2,2)$, $(2,2,2,2)$, or $(3,2,2,1)$.  Note if $v$ shares a short string with a poor vertex $v'$ of weight 6, then $v$ pays $v'$ only 1 and each of the other $(p-1)$ vertices in $A$ at most 3. Thus $v$ pays the vertices in $A$ at most $3(p-1)+1$, and so it follows that $ch_3(v) \geq 4-p$. Since $p \in \{3,4\}$, this is non-negative.  Thus we may assume that $v$ pays at least three poor vertices of weight 7. We finish the proof after the following two claims. \\

\noindent{
\textbf{Claim 1.}\emph{ $v$ is not of type $(2,2,2,2)$.}
}
\begin{proof}
Suppose not. Recall $v$ only pays a subset of the vertices of degree 3 that are not contained in cells. Furthermore, since $v$ is of type $(2,2,2,2)$, $v$ only pays vertices that are incident with 2-strings. Since the only vertices of degree 3 and weight 7 not contained in cells and incident with 2-strings are vertices of type $(3,2,2)$, we may assume $v$ pays at least three vertices of type $(3,2,2)$. Since each vertex of type $(3,2,2)$ in $A$ is not in a cell, by Lemma \ref{3str} they are each contained in 9-cycles. Thus there exist vertices $a \neq b$ in $A$ of type $(3,3,2)$, such that $S_a= aa_1a_2v$ is a 2-string shared by $v$ and $a$, $S_b$ is a 2-string shared by $v$ and $b$, and $S_{ab}$ is a 2-string shared by $a$ and $b$. Similarly, there exist vertices $c \neq d$ in $A$ such that $c$ is of type $(3,2,2)$, $S_c$ is a 2-string shared by $v$ and $c$, $S_d$ is a 2-string shared by $v$ and $d$, and $S_{cd}$ is a 2-string shared by $c$ and $d$. Let $C$ be the 9-cycle formed by $S_a\cup S_b\cup S_{ab}$.  Let $G'$ be the graph obtained from $G$ by identifying $v$ and $a_1$ to a new vertex $z$, and deleting $a_2$.  Note $G'$ does not admit a homomorphism $\phi$ to $C_7$ as otherwise $\phi$ extends to $G$ by setting $\phi(v) = \phi(a_1) = \phi(z)$ and $\phi(a_2) \in N_{C_7}(\phi(z))$. Furthermore, $G'$ itself is not $C_7$-critical, as the cycle $C'$ obtained from $C$ with the identification of $v$ and $a_1$ is a cell of weight 10 in $G'$, contradicting Lemma \ref{cell-weight}. Thus $G'$ contains a $C_7$-critical subgraph $G''$.

Note $G''$ does not contain at least one string incident with $C'$ or one string in $E(C')$. We claim one of these strings not contained in $G''$ contains at least one internal vertex. To see this, suppose not. Note $C'$ is incident with four strings: $S_c$, $S_d$, and two 3-strings $S_a'$ and $S_b'$ incident with $a$ and $b$, respectively. Since these four strings all have internal vertices, we may assume they are all contained in $G''$. Since $S_{ab}$ and $S_b$ are both 2-strings, we may assume they are both contained in $G''$. Thus $G'$ does not contain the edge $az$. But then $S_{ab} \cup S_a'$ is a 6-string contained in $G''$, contradicting Lemma \ref{max-str}.

Thus we may assume there is a vertex in $V(G) \setminus \{a_1, a_2, v\}$ that is not contained in $V(G'')$.

Since $G'' \not \subset G$, it follows that $G''$ contains the new vertex $z$. Note $G''$ does not contain a triangle or 5-cycle, since $S_a$ is contained in a 9-cycle in $G$ and so is not contained in a cell by Lemma \ref{7and9cycles}. Since $v(G'') < v(G')$, it follows that $p(G'') \leq 2$.

Let $F$ be the graph obtained from $G''$ by splitting $z$ back into the vertices $v$ and $a_1$ and adding in the path $va_2a_1$. Then $p(F) = p(G'') + 17(2)-15(2) \leq 6$. But since $F \neq G$, this contradicts Lemma \ref{potentials7}.
\end{proof}

\noindent{
\textbf{Claim 2.} \emph{$v$ is not a vertex of type $(3,2,2,2)$.}
}
\begin{proof}
Suppose not.  Since $v$ is of type $(3,2,2,2)$, $v$ only pays vertices that are incident with 2-strings. Since the only vertices of degree 3 and weight 7 not contained in cells and incident with 2-strings are vertices of type $(3,2,2)$, we may assume $v$ pays three vertices of type $(3,2,2)$. Since each vertex of type $(3,2,2)$ in $A$ is not in a cell, by Lemma \ref{3str} they are each contained in 9-cycle. Let $a \neq b \neq c$ be the three vertices of type $(3,2,2)$ in $A$. Let $S_a$, $S_b$, and $S_c$ be the 2-strings shared by $v$ with $a$, $b$, and $c$, respectively.   Let $S_a'$ be the other 2-string incident with $a$. By Lemma \ref{3str}, $S_a' \cup S_a$ is contained in a 9-cycle $C$. Since $v$ has degree exactly 4 and is incident only with 2-strings and a 3-string, one of $S_b$ and $S_c$ is contained in $C$. Without loss of generality, we may assume $C = S_a \cup S_a' \cup S_b$, and so that $S_a'$ is shared by $b$ and $a$.  Similarly, let $S_c' \neq S_c$ be a 2-string incident with $c$. By Lemma \ref{3str}, $S_c \cup S_c'$ is contained in a 9-cycle $C' \neq C$. Since $v$ has degree exactly 4 and is incident only with 2-strings and a 3-string, it follows that there exists a 2-string $S \neq S_c$ incident with $v$ that is contained in $C'$. Thus $S_a$ or $S_b$ is contained in $C'$. But this contradicts Lemma \ref{9cycles}, since $S_b \cup S_a \subset C$.
\end{proof}

The only remaining possibility is then that $v$ is of type (3,2,2,1). Since $|A| \geq 3$ and each vertex in $A$ has weight 7, we may assume $v$ shares its incident 2-strings $S_1$ and $S_2$ with two poor $(3,2,2)$ vertices $u_1$ and $u_2$, and its incident 1-string $S_3$ with a poor $(3,3,1)$ vertex $u_3$. By applying Lemma \ref{3str} to $u_1, u_2$, and $u_3$, we find each of $S_1$ and $S_2$ is contained in two 9-cycles, contradicting Lemma \ref{9cycles}.
\end{proof}

\begin{lemma}\label{Step4}
Let $v \in V(G)$ be a vertex of degree 3 and weight at most 4. At the end of Step 4, $ch_4(v) \geq 0$.
\end{lemma}
\begin{proof}
Suppose not. Let $A$ be the set of vertices that are poor immediately after Step 3 and that each share a short string with $v$. Let $p = |A|$. Note none of the $p$ vertices in $A$ are contained in cells, as otherwise they have charge at least 0 at the end of Step 1. Furthermore, $v$ is not contained in a cell, as otherwise this cell sends charge to the vertices in $A$ in Step 2.  Note since all vertices of degree 3 and weight 8 are contained in cells by Lemma \ref{4str}, each vertex in $A$ has weight at most 7, and so $v$ sends at most $3p$ units of charge in Step 3.

Note if $A$ contains $p \leq 3$ vertices of weight 6, then $v$ sends only 3 units of charge in Step 4, and so
\begin{align*}
ch_4(v) &\geq ch_3(v) - 3 \\
&= 15\deg(v)-2\textrm{wt}(v)-34-3\\
&\geq 45-8-34-3 \\
&=0, \textnormal{ a contradiction.}
\end{align*} 

We may therefore assume $A$ contains a vertex $u$ of weight at least 7. Since $A$ only contains vertices of weight at most seven, we may assume $u$ has weight exactly 7. Furthermore, if $p=1$ (and so $v$ sends charge to only one vertex in Step 4), then $ch_4(v) = ch_3(v)-3 =0$, a contradiction. Thus we may assume $p \geq 2$. Since vertices of type $(4,3,0)$ and $(4,2,1)$ are contained in cells and thus have charge 0 after Step 1, we may assume $u$ is either of type $(3,3,1)$ or of type $(3,2,2)$.

First suppose $u$ is of type $(3,2,2)$. Let $u_1$ be the other vertex that shares a 2-string with $u$. Note if $u_1$ is contained in a cell, then $ch_3(u) = 0$ by Step 2, and so $u \not \in A$, a contradiction. Furthermore, if $\deg(u_1) \geq 4$, then $ch_3(u) = 0$ by Step 3, and so again $u \not \in A$. We may therefore assume $u_1$ has degree three. But then this contradicts Lemma \ref{3,2,2s}, as neither $v$ nor $u_1$ is contained in a cell.

Thus we may assume $u$ is of type $(3,3,1)$. Let $S_1$ and $S_3$ be the two 3-strings incident with $u$, and let $S_2$ be the 1-string shared by $u$ and $v$ (see Figure \ref{fig:deg3wt4}). 
By Lemma \ref{3str}, since $u$ is a $(3,3,1)$-vertex not contained in a cell,  we have that $S_2 \cup S_3$ is contained in a 9-cycle $C$. Similarly, $S_1 \cup S_2$ is contained in a 9-cycle $C' \neq C$. Since both 9-cycles contain $S_2$, we have that $C' \cap C = S_2$ by Lemma \ref{9cycles}. 
Let $S_4$ and $S_5$ be the other two strings incident with $v$. Without loss of generality, suppose $S_4 \subset C'$ and $S_5 \subset C$. Note since $p \geq 2$, one of $S_4$ and $S_5$ is shared by $v$ with a vertex in $A$. Without loss of generality, we may assume $S_4$ is shared by $v$ and a vertex $w$ in $A$. Since $v$ has weight at most four and is incident with a 1-string $S_2$, it follows that at least one of $S_4$ and $S_5$ is not a 2-string.

First suppose $S_4$ is not a 2-string.  Since $w \in A$, we have that $w$ has degree three and $\textrm{wt}(w) \geq 6$. Furthermore, since $w$ is not contained in a cell, $w$ is not incident with a 4-string by Lemma \ref{4str}. Since the internal vertices of two of the strings incident with $w$ are contained in $C' \setminus V(S_1 \cup S_2)$, together these two strings contribute at most 1 to the weight of $w$. But then $w$ has weight at most 4, a contradiction.

Thus we may assume $S_4$ is a 2-string, and since $w$ has weight at least 6, we have that $w$ is a vertex of type $(3,2,k)$ with $k \geq 1$. 

Let $S_6$ denote the third string incident with $w$ (so $w$ is incident with $S_1, S_4$, and $S_6$). By Lemma \ref{3str} applied to $w$, $S_4$ and $S_6$ are contained in a 9-cycle $C''$.  Since $S_4$ is contained in $C' \cap C''$, this contradicts Lemma \ref{9cycles}.
\end{proof}

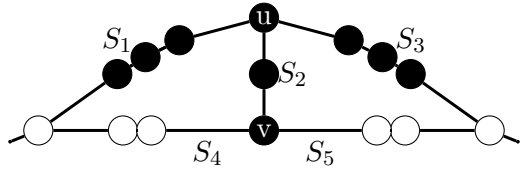
\begin{figure}[!h]
	\begin{center}
        
            \begin{tikzpicture}[scale=0.75]	
			\tikzset{black node/.style={shape=circle,draw=black,fill=black,inner sep=0pt, minimum size=11pt}}
            \tikzset{white node/.style={shape=circle,draw=black,fill=white,inner sep=0pt, minimum size=11pt}}
			\tikzset{invisible node/.style={shape=circle,draw=black,fill=black,inner sep=0pt, minimum size=1pt}}
				\tikzset{edge/.style={black, line width=.4mm}}		

                \node[black node] (1) at (0,0){$\color{white}v$};
                \node[black node] (2) at (0,1){};
                \node[black node] (3) at (0,2){$\color{white}u$};                
                \node[black node] (4) at (1.5,1.6){};                
                \node[black node] (5) at (2.1,1.3){};                
                \node[black node] (6) at (2.6,1){};                
                \node[white node] (7) at (4.0,0){};                
                \node[white node] (8) at (2.5,0){};                
                \node[white node] (9) at (2,0){};                
                \node[black node] (10) at (-1.5,1.6){};                
                \node[black node] (11) at (-2.1,1.3){};                
                \node[black node] (12) at (-2.6,1){};                
                \node[white node] (13) at (-4.0,0){};                
                \node[white node] (14) at (-2.5,0){};                
                \node[white node] (15) at (-2,0){};
                \node[invisible node] (16) at (4.5, -0.2){};
                \node[invisible node] (17) at (-4.5, -0.2){};

                \draw[edge] (1)--(2);
                \draw[edge] (2)--(3);
                \draw[edge] (3)--(4);
                \draw[edge] (4)--(5);
                \draw[edge] (5)--(6);
                \draw[edge] (6)--(7);
                \draw[edge] (7)--(8);
                \draw[edge] (8)--(9);
                \draw[edge] (9)--(1);
                \draw[edge] (3)--(10);
                \draw[edge] (10)--(11);
                \draw[edge] (11)--(12); 
                \draw[edge] (12)--(13);
                \draw[edge] (13)--(14);
                \draw[edge] (14)--(15); 
                \draw[edge] (15)--(1);
                \draw[edge] (10)--(11);
                \draw[edge] (11)--(12);   
                \draw[edge] (7)--(16);
                \draw[edge] (13)--(17);
                \node[] at (-2.6,1.6) {$S_1$};
                \node[] at (0.6,1) {$S_2$};
                \node[] at (2.6,1.6) {$S_3$};
                \node[] at (-1,-0.4) {$S_4$};   
                \node[] at (1,-0.4) {$S_5$};                 
			\end{tikzpicture}
            \hskip 8mm
    \end{center}
	\caption{Figure for Lemma \ref{Step4}. $v$ has weight at most four and degree three. The white vertices are of unknown degree, though their degree is at least that shown. The black vertices' degrees are as illustrated.}
	\label{fig:deg3wt4}
\end{figure}
\begin{lemma}\label{Step5}
Let $v \in V(G)$ be a vertex of degree 3 and weight 5 that shares a short string with only one poor vertex either during Step 5. At the end of Step 5, $ch_5(v) \geq 0$.
\end{lemma}
\begin{proof}
Suppose not. Let $u$ be the poor vertex with that shares a short string with $v$ during Step 5. Suppose $\textrm{wt}(u) = 6$. Then $ch_4(u) = -1$, and so \begin{align*}
ch_5(v) &= ch_4(v)-ch_4(u) \\
&=15\deg(v)-2\textrm{wt}(v)-34-1 \\
&= 45-10-34-1 \\
&=0, \textnormal{ a contradiction.}
\end{align*} We may therefore assume $u$ has weight at least seven. Note since $ch_4(u) < 0$, $u$ is not contained in a cell by the discharging rules. By Lemma \ref{4str}, it follows that $u$ is not of type (4,2,2), (4,3,0), or (4,2,1). Thus we may assume $u$ is a vertex of type either (3,3,1), or (3,2,2). By Lemma \ref{3,2,2s}, if $u$ is a (3,2,2)-vertex it has charge at least 0 by either Step 1, 2, or 3. 

Therefore we may assume $u$ is a vertex of type (3,3,1), and hence the string shared by $u$ and $v$ is a 1-string. Note $v$ is not incident with a 4-string as otherwise it is contained in a cell by Lemma \ref{deg4s} and so $ch_2(u) \geq 0$, a contradiction. Thus since $v$ has degree 3 and weight 5 and is incident with a 1-string, it is either a vertex of type (3,1,1) or of type (2,2,1). 

First suppose $v$ is of type $(3,1,1)$.  Let $S_1$ and $S_2$ be the two 1-strings incident with $v$, named such that $S_1$ is shared by $u$ and $v$. Let $S_3$ be the 3-string incident with $v$. Let the two 3-strings incident with $u$ be named $S_4$ and $S_5$. Note $S_4$ and $S_5$ do not have the same two endpoints by Lemma \ref{par-str}. Furthermore $v$ and $u$ share only $S_1$ by Lemma \ref{par-str}.  By Lemma \ref{3str} applied to $u$ and $S_4$, since $u$ is not contained in a cell $S_1 \cup S_5$ is contained in a 9-cycle $C$. Note since $v$ has degree 3, it follows that either $S_2$ or $S_3$ is contained in $C$. Since $v(S_1 \cup S_5 \cup S_3) = 11$, we may assume that $S_2 \subset C$. Similarly, by Lemma \ref{3str} applied to $u$ and $S_5$, we have that $S_4 \cup S_1$ is contained in a 9-cycle $C'$. Note since $v$ has degree 3, it follows that either $S_2$ or $S_3$ is contained in $C'$. Since $v(S_1 \cup S_4 \cup S_3) = 11$, it follows that $S_2 \subset C'$.  Since $S_1 \cup S_2 \subset C \cap C'$, this contradicts Lemma \ref{9cycles}.

Therefore we may assume $v$ is a $(2,2,1)$-vertex. Let $a$ and  $b$ be the two vertices that share a 2-string with $v$. Note $a \neq b$ by Lemma \ref{par-str}. Let $S_a$ be the string shared by $a$ and $v$, and let $S_b$ be the 2-string shared by $b$ and $v$. Let $S_1$, $S_2$ and $S_3$ be the three strings incident with $u$, named such that $S_1$ is shared by $u$ and $v$. Note $S_2$ and $S_3$ do not have the same endpoints by Lemma \ref{par-str}. By Lemma \ref{3str} applied to $u$ and $S_2$, since $u$ is not contained in a cell we have that $S_3 \cup S_1$ is contained in a 9-cycle $C$.  Similarly, by applying Lemma \ref{3str} to $u$ and $S_3$, we have that $S_2 \cup S_1$ is contained in a 9-cycle $C'$. Thus without loss of generality, we may assume $S_2$ is shared by $a$ and $u$, and that $S_3$ is shared by $b$ and $u$.

Let $G'$ be the graph obtained from $G$ by deleting $u$, $v$, and all of the internal vertices of  their incident strings. Since $G$ is $C_7$-critical, $G'$ has a homomorphism $\phi$ to a cycle $C = c_1c_2c_3c_4c_5c_6c_7c_1$. Without loss of generality, we may assume $\phi(a) = c_1$ and $\phi(b) \in \{c_1, c_2, c_3, c_4\}$. Since $\phi$ does not extend to $G$, $\phi(b) = c_4$. (To see the extensions of all other homomorphisms to $G$, see Figure \ref{fig:deg3wt5}.) Let $G''$ be the graph obtained from $G'$ by adding a new vertex $z$ and edges $az$ and $bz$. Note now the following: \\

\noindent{
\textbf{Claim 1.} \emph{There does not exist a homomorphism $\phi: G'' \rightarrow C$ with $\phi(a) = c_1$ and $\phi(b) = c_4$}.
}
\begin{proof}
Suppose $\phi : G'' \rightarrow C$ is such that $\phi(a) = c_1$. Note $\phi(b) \in B_\phi(b |a, azb)$. But $B_\phi(b | a, azb) = N_C(N_C(c_1)) = \{c_1, c_3, c_6\}$.
\end{proof}

Thus if $G''$ admits a homomorphism to $C$, this homomorphism extends to a homomorphism of $G$ to $C$, since by Claim 1 there does not exist a homomorphism $\phi$  from $G''$ to $C$ with $\phi(a)= c_1$ and $\phi(b)=c_4$. We may thus assume $G''$ contains a $C_7$-critical subgraph $G'''$, and since $G''' \not \subset G$, we have $z \in V(G''')$. Furthermore, since $G'''$ has minimum degree at least two, $\{az, zb\} \in E(G''')$. 

Suppose $G'''$ is a triangle. Then $ab$ is an edge in $E(G)$. But then $abS_aS_b$ is a cell $C''$ with $S_a \subset C'' \cap C'$, contradicting Lemma \ref{7and9cycles}. Suppose now $G'''$ is a 5-cycle. Then there exists an $(a,b)$-path $P$ of length 3. But then $P\cup S_a \cup S_b$ is a 9-cycle $C''$ with $S_a \subset C'' \cap C'$, contradicting Lemma \ref{9cycles}.

Thus we may assume that $G'''$ is not a triangle or 5-cycle. Since $v(G''') < v(G)$, it follows that $G'''$ is not a counterexample to Theorem \ref{main}, and so $p(G''') \leq 2$. Let $F$ be the graph obtained from $G'''$ by deleting $z$ and adding $S_a \cup S_b$. Then $p(F) = p(G''') + 17(4)-15(4) \leq 10$. But since $u$ is a vertex of degree 3 and $u \not \in V(F)$, this contradicts Lemma \ref{potentials7}.
\end{proof}
\begin{figure}
	\begin{center}
        
            \begin{tikzpicture}[scale=0.75]	
			\tikzset{black node/.style={shape=circle,draw=black,fill=black,inner sep=0pt, minimum size=11pt}}
            \tikzset{white node/.style={shape=circle,draw=black,fill=white,inner sep=0pt, minimum size=11pt}}
			\tikzset{invisible node/.style={shape=circle,draw=black,fill=black,inner sep=0pt, minimum size=1pt}}
				\tikzset{edge/.style={black, line width=.4mm}}		

                \node[black node] (1) at (0,0){\color{white}$c_4$};
                \node[black node] (2) at (0,1){\color{white}$c_5$};
                \node[black node] (3) at (0,2){\color{white}$c_4$};                
                \node[black node] (4) at (1.5,1.6){\color{white}$c_5$};                
                \node[black node] (5) at (2.1,1.3){\color{white}$c_6$};                
                \node[black node] (6) at (2.6,1){\color{white}$c_7$};                
                \node[white node] (7) at (4.0,0){\color{black}$c_1$};                
                \node[black node] (8) at (2.5,0){\color{white}$c_2$};                
                \node[black node] (9) at (2,0){\color{white}$c_3$};                
                \node[black node] (10) at (-1.5,1.6){\color{white}$c_5$};                
                \node[black node] (11) at (-2.1,1.3){\color{white}$c_6$};                
                \node[black node] (12) at (-2.6,1){\color{white}$c_7$};                
                \node[white node] (13) at (-4.0,0){\color{black}$c_1$};               
                \node[black node] (14) at (-2.5,0){\color{white}$c_2$};                
                \node[black node] (15) at (-2,0){\color{white}$c_3$};
                \node[invisible node] (16) at (4.5, -0.2){};
                \node[invisible node] (17) at (-4.5, -0.2){};

                \draw[edge] (1)--(2);
                \draw[edge] (2)--(3);
                \draw[edge] (3)--(4);
                \draw[edge] (4)--(5);
                \draw[edge] (5)--(6);
                \draw[edge] (6)--(7);
                \draw[edge] (7)--(8);
                \draw[edge] (8)--(9);
                \draw[edge] (9)--(1);
                \draw[edge] (3)--(10);
                \draw[edge] (10)--(11);
                \draw[edge] (11)--(12); 
                \draw[edge] (12)--(13);
                \draw[edge] (13)--(14);
                \draw[edge] (14)--(15); 
                \draw[edge] (15)--(1);
                \draw[edge] (10)--(11);
                \draw[edge] (11)--(12);   
                \draw[edge] (7)--(16);
                \draw[edge] (13)--(17);
                \node[] at (0,2.5) {$u$};
                \node[] at (0,-0.5) {$v$};
                \node[] at (-4.5,0.2) {$a$};   
                \node[] at (4.5,0.2) {$b$};                 
			\end{tikzpicture}
             \begin{tikzpicture}[scale=0.75]	
			\tikzset{black node/.style={shape=circle,draw=black,fill=black,inner sep=0pt, minimum size=11pt}}
            \tikzset{white node/.style={shape=circle,draw=black,fill=white,inner sep=0pt, minimum size=11pt}}
			\tikzset{invisible node/.style={shape=circle,draw=black,fill=black,inner sep=0pt, minimum size=1pt}}
				\tikzset{edge/.style={black, line width=.4mm}}		

                \node[black node] (1) at (0,0){\color{white}$c_5$};
                \node[black node] (2) at (0,1){\color{white}$c_6$};
                \node[black node] (3) at (0,2){\color{white}$c_5$};                
                \node[black node] (4) at (1.5,1.6){\color{white}$c_6$};                
                \node[black node] (5) at (2.1,1.3){\color{white}$c_7$};                
                \node[black node] (6) at (2.6,1){\color{white}$c_1$};                
                \node[white node] (7) at (4.0,0){\color{black}$c_2$};                
                \node[black node] (8) at (2.5,0){\color{white}$c_3$};                
                \node[black node] (9) at (2,0){\color{white}$c_4$};                
                \node[black node] (10) at (-1.5,1.6){\color{white}$c_4$};                
                \node[black node] (11) at (-2.1,1.3){\color{white}$c_3$};                
                \node[black node] (12) at (-2.6,1){\color{white}$c_2$};                
                \node[white node] (13) at (-4.0,0){\color{black}$c_1$};               
                \node[black node] (14) at (-2.5,0){\color{white}$c_7$};                
                \node[black node] (15) at (-2,0){\color{white}$c_6$};
                \node[invisible node] (16) at (4.5, -0.2){};
                \node[invisible node] (17) at (-4.5, -0.2){};

                \draw[edge] (1)--(2);
                \draw[edge] (2)--(3);
                \draw[edge] (3)--(4);
                \draw[edge] (4)--(5);
                \draw[edge] (5)--(6);
                \draw[edge] (6)--(7);
                \draw[edge] (7)--(8);
                \draw[edge] (8)--(9);
                \draw[edge] (9)--(1);
                \draw[edge] (3)--(10);
                \draw[edge] (10)--(11);
                \draw[edge] (11)--(12); 
                \draw[edge] (12)--(13);
                \draw[edge] (13)--(14);
                \draw[edge] (14)--(15); 
                \draw[edge] (15)--(1);
                \draw[edge] (10)--(11);
                \draw[edge] (11)--(12);   
                \draw[edge] (7)--(16);
                \draw[edge] (13)--(17);
                \node[] at (0,2.5) {$u$};
                \node[] at (0,-0.5) {$v$};
                \node[] at (-4.5,0.2) {$a$};   
                \node[] at (4.5,0.2) {$b$};
			\end{tikzpicture}
            \vskip 6mm
             \begin{tikzpicture}[scale=0.75]	
			\tikzset{black node/.style={shape=circle,draw=black,fill=black,inner sep=0pt, minimum size=11pt}}
            \tikzset{white node/.style={shape=circle,draw=black,fill=white,inner sep=0pt, minimum size=11pt}}
			\tikzset{invisible node/.style={shape=circle,draw=black,fill=black,inner sep=0pt, minimum size=1pt}}
				\tikzset{edge/.style={black, line width=.4mm}}		

                \node[black node] (1) at (0,0){\color{white}$c_7$};
                \node[black node] (2) at (0,1){\color{white}$c_6$};
                \node[black node] (3) at (0,2){\color{white}$c_5$};                
                \node[black node] (4) at (1.5,1.6){\color{white}$c_4$};                
                \node[black node] (5) at (2.1,1.3){\color{white}$c_3$};                
                \node[black node] (6) at (2.6,1){\color{white}$c_2$};                
                \node[white node] (7) at (4.0,0){\color{black}$c_3$};                
                \node[black node] (8) at (2.5,0){\color{white}$c_2$};                
                \node[black node] (9) at (2,0){\color{white}$c_1$};                
                \node[black node] (10) at (-1.5,1.6){\color{white}$c_4$};                
                \node[black node] (11) at (-2.1,1.3){\color{white}$c_3$};                
                \node[black node] (12) at (-2.6,1){\color{white}$c_2$};                
                \node[white node] (13) at (-4.0,0){\color{black}$c_1$};               
                \node[black node] (14) at (-2.5,0){\color{white}$c_2$};                
                \node[black node] (15) at (-2,0){\color{white}$c_1$};
                \node[invisible node] (16) at (4.5, -0.2){};
                \node[invisible node] (17) at (-4.5, -0.2){};

                \draw[edge] (1)--(2);
                \draw[edge] (2)--(3);
                \draw[edge] (3)--(4);
                \draw[edge] (4)--(5);
                \draw[edge] (5)--(6);
                \draw[edge] (6)--(7);
                \draw[edge] (7)--(8);
                \draw[edge] (8)--(9);
                \draw[edge] (9)--(1);
                \draw[edge] (3)--(10);
                \draw[edge] (10)--(11);
                \draw[edge] (11)--(12); 
                \draw[edge] (12)--(13);
                \draw[edge] (13)--(14);
                \draw[edge] (14)--(15); 
                \draw[edge] (15)--(1);
                \draw[edge] (10)--(11);
                \draw[edge] (11)--(12);   
                \draw[edge] (7)--(16);
                \draw[edge] (13)--(17);
                \node[] at (0,2.5) {$u$};
                \node[] at (0,-0.5) {$v$};
                \node[] at (-4.5,0.2) {$a$};   
                \node[] at (4.5,0.2) {$b$};              
			\end{tikzpicture}

    \end{center}
	\caption{Figure for Claim \ref{Step5}. Extensions of $\phi$ to $G$. The white vertices are of unknown degree, though their degree is at least that shown. The black vertices' degrees are as illustrated.}
	\label{fig:deg3wt5}
\end{figure}
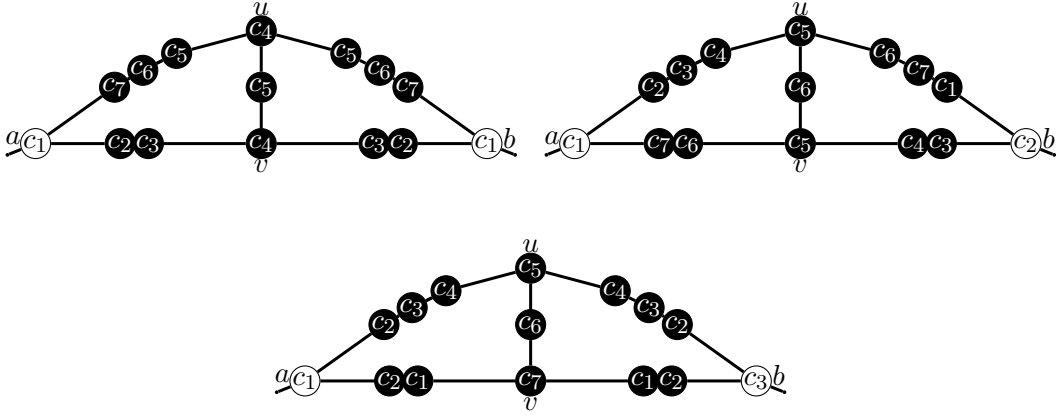

\subsection{All Poor Structures Receive Charge}\label{getpaid}
We have shown that no new poor vertices or cells are created through discharging. Since no cell is initially poor, it remains to show only the following lemma.

\begin{lemma}\label{getmoneygetpaid}
Let $v \in V(G)$ have $ch_0(v) < 0$. At the end of Step 5, $ch_5(v) \geq 0$. 
\end{lemma}
\begin{proof}
Suppose not. Since $ch_0(v) < 0$, we may assume $v$ is a vertex of degree 3 and weight at least 6.  We may assume $v$ is not contained in a cell, as otherwise $v$ receives $-ch_0(v)$ charge from its cell in Step 1, resulting in $ch_5(v) = ch_1(v) =0$. Since $v$ is not in a cell, by Lemma \ref{4str} $v$ is not incident with a 4-string.  It follows that $v$ is not of type $(4,2,2)$, $(4,3,0), (4,2,1), (4,2,0)$, or $(4,1,1)$. We may also assume $v$ does not share a short string with a vertex in a cell or a vertex of degree at least 4, as otherwise $v$ receives charge at least $-ch_0(v)$ in either Step 2 or Step 3, a contradiction. Furthermore, we may assume $v$ does not share a short string with a vertex of degree 3 of weight at most 4, as otherwise $ch_5(v) \geq 0$ by Step 4. Finally, by Step 5 we may assume that if $v$ shares a short string with a vertex $u$ of degree 3 and weight 5, then $v$ is not the only poor vertex that shares a short string with $u$. 

Note $v$ is not a vertex of type $(3,2,2)$. To see this, suppose not. Let $a$ and $b$ be the vertices that share a short string with $v$. Note $a\neq b$ by Lemma \ref{par-str}. Since $ch_5(v) < 0$, neither $a$ nor $b$ is contained in a cell, and both $a$ and $b$ have degree 3.  But by Lemma \ref{3,2,2s}, since $v$ is not contained in a cell and $\deg(a) = \deg(b) = 3$, at least one of $a$ and $b$ is contained in a cell, a contradiction.

Note furthermore the following claim.
\begin{claim} \label{(2,2,2)}
$v$ is not of type $(2,2,2)$.
\end{claim} 
\begin{proof}

Suppose not. By Lemma \ref{2,2,2s}, either $v$ is contained in a cell, $v$ shares a short string with a cell, or $v$ shares a short string with a vertex of degree at least 4. But then $v$ receives charge in either Step 1, Step 2, or Step 3, contradicting that $ch_5(v) < 0$.
\end{proof}

Thus $v$ is either a vertex of type (3,3,0), (3,2,1), or (3,3,1). 

\begin{claim}
$v$ is not of type (3,3,1).
\end{claim} 
\begin{proof}
Suppose not, and let $u$ be the vertex with which $v$ shares its short string. Since $ch_5(v) < 0$, it follows that $u$ is either a vertex of degree 3 and weight at least 6, or a vertex of degree 3 and weight 5 that shares short strings with at least two vertices that are poor after Step 5. Let $A$ be the set of vertices that share a short string with $u$ and that have negative charge after Step 5.

Let $S_1$ and $S_3$ be the 3-strings incident with $v$, and let $S_2$ be the 1-string shared by $u$ and $v$. Let $a$ and $b$ be the other endpoints of the strings $S_1$ and $S_3$, respectively. Note $a \neq b$ by Lemma \ref{par-str}. By Lemma \ref{3str} applied to $v$, since $v$ is not contained in a cell by assumption, $S_2 \cup S_3$ is contained in a 9-cycle $C_1$.  Similarly, $S_2 \cup S_1$ is contained in a 9-cycle $C_2$. Let $S_4 \neq S_2$ and $S_5 \not \in \{S_2, S_4\}$ be the other two strings adjacent with $u$, such that at least one edge from $S_5$ contained in $C_1$. Note $C_2$ does not contain an edge $e$ in $S_5$ as otherwise $eS_2 \subset C_1 \cap C_2$, contradicting Lemma \ref{9cycles}. Since $u$ has degree 3, and the distance from each of $a$ and $b$ to $u$ along $C_2$ and $C_1$, respectively, is at most three, we have that $u$ has weight at most five. Since $\textrm{wt}(u) \geq 5$, we have therefore that $\textrm{wt}(u) = 5$, and so that $S_4$ and $S_5$ are 2-strings with endpoints $a, u$ and $b,u$, respectively. Furthermore, since $|A| \geq 2$, at least one of $a$ and $b$ has negative final charge.  Without loss of generality, we may assume $a \in A$, and so that $a$ has degree 3 and $\textrm{wt}(a) \geq 6$. Let $S_4$ be the 2-string shared with $u$ by $a$. Let $S_6 \not \in \{S_1, S_4\}$ be the third string incident with $a$. Since $a$ has weight at least six, since $S_1$ is a 3-string, and since $S_4$ is a 2-string, we have that $S_6$ is a $k$-string with $k \geq 1$. But then Lemma \ref{3str} applies to $a$, and so $S_6$ and $S_4$ are contained in a 9-cycle $C_3$. Since $S_4 \subset C_3 \cap C_1$, this contradicts Lemma \ref{9cycles}.  We may therefore assume $v$ is not a vertex of type $(3,3,1)$. 
\end{proof}

\begin{claim}\label{(3,2,1)}
 $v$ is not a vertex of type (3,2,1). 
\end{claim} 
\begin{proof}
Suppose not. Let $a$, $b$, and $c$ be the vertices that share a 1-string, 2-string, and 3-string, respectively, with $v$. Let $S_a, S_b$ and $S_c$ be the three strings incident with $v$, such that $S_a$ is incident with $a$, $S_b$ is incident with $b$, and $S_c$ is incident with $c$. Note $a \neq b$ and $a \neq c$ since $G$ has girth at least 7 by Lemma \ref{girth}. Furthermore, $b \neq c$ since $v$ is not contained in a cell by assumption. By Lemma \ref{3str}, $S_b \cup S_a$ is contained in a 9-cycle $C$. Let $S_{ab}$ be the $(a,b)$-path of length 4 in $C$ with $v \not \in V(S_{ab})$.  Since $ch_5(v) < 0$, it follows from the discharging rules that each of $a$ and $b$ has degree 3, is not contained in a cell, and has weight at least 5.

Suppose first that $S_{ab}$ is a 3-string. Since $a$ has weight at least 5, it is incident with a $k$-string $S_d$, with $k \geq 1$. By Lemma \ref{3str}, $S_a \cup S_d$ is contained in a 9-cycle $C' \neq C$. Note $E(S_b) \cap E(C') = \emptyset$, as otherwise $S_a\cup S_b \subset C \cap C'$, contradicting Lemma \ref{9cycles}. Thus we may assume $S_c$ is contained in $C'$, and so $k \leq 2$. 

Suppose first that $k=2$. Let $G' = G \setminus (V(S_a \cup S_b\cup S_c\cup S_d\cup S_{ab}) \setminus \{b, c\})$. Since $G' \subsetneq G$ and $G$ is $C_7$-critical, $G'$ admits a homomorphism $\phi$ to $C = c_1c_2c_3c_4c_5c_6c_7c_1$. Without loss of generality, we may assume $\phi(b) = c_1$ and $\phi(c) \in \{c_1, c_2, c_3, c_4\}$. Note $\phi(c) = c_3$ as otherwise $\phi$ extends to $G$. To see this, see Figure \ref{fig:321s}. Let $G_1 \in P_3(G')$ be the graph obtained from $G'$ by adding a $(b,c)$-path $P$ of length $3$. \\

\noindent{
\textbf{Claim 1.}\emph{ There does not exist a homomorphism $\phi: G_1 \rightarrow C$ with $\phi(b) = c_1$ and $\phi(c) = c_3$.}
}
\begin{proof}
Suppose $\phi : G_1 \rightarrow C$ is such that $\phi(b) = c_1$. Note $\phi(c) \in B_\phi(c |b, P)$. But $B_\phi(c | b, P) = N_C(N_C(N_C(c_1)))$, and $N_C(N_C(N_C(c_1))) = \{c_2, c_7, c_4, c_5\}$.
\end{proof}

By Claim 1, it follows that $G_1$ does not admit a homomorphism to $C_7$. Therefore $G_1$ contains a $C_7$-critical subgraph $G_2$. Since $G_2 \not \subset G$ and $G_2$ has minimum degree 2, $P \subset G_2$.  Note since $b \neq c$, it follows that $G_2$ is not a triangle.  Suppose $G_2$ is a 5-cycle. Then there exists a $(b,c)$-path $Q$ of length 2 in $G$. Let $F = Q \cup S_a \cup S_b \cup S_c \cup S_d \cup S_{ab}$. Note $v(F) = 16$ and $e(F) = 18$, and so it follows that $p(F) = 17(16)-15(18) = 2$. This contradicts Lemma \ref{potentials7}.  We may therefore assume that $G_2$ is not a triangle or 5-cycle. Since $v(G_2) < v(G)$ and $G$ is a minimum counterexample, it follows that $p(G_2) \leq 2$. Let $F$ be the graph obtained from $G_2$ by deleting $P \setminus \{a,b\}$ and adding $S_d \cup S_{ab}$. Then $p(F)  = p(G_2) + 17(4)-15(4) \leq 10$. By Lemma \ref{potentials7}, either $F = G$ or $G \in P_5(F)$. But since $v$ is not contained in $F$ and $\deg(v)=3$, this is a contradiction. 

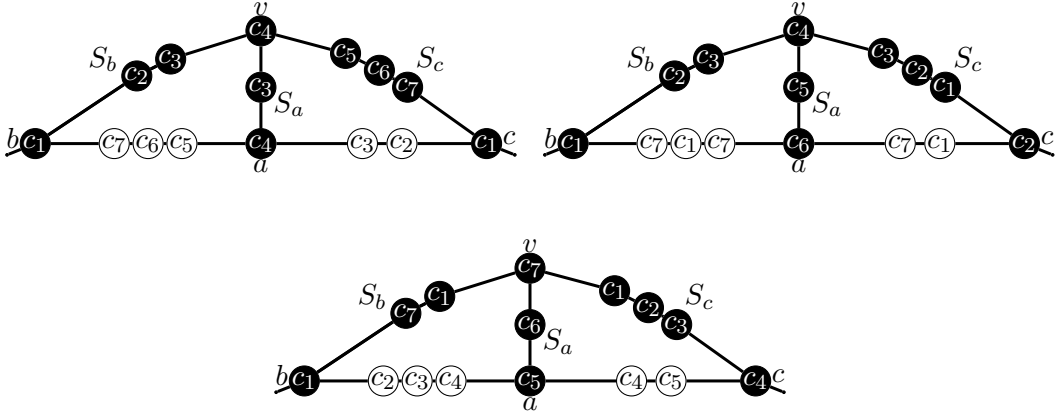
\begin{figure}
	\begin{center}
            \begin{tikzpicture}[scale=0.75]	
			\tikzset{black node/.style={shape=circle,draw=black,fill=black,inner sep=0pt, minimum size=11pt}}
            \tikzset{white node/.style={shape=circle,draw=black,fill=white,inner sep=0pt, minimum size=11pt}}
			\tikzset{invisible node/.style={shape=circle,draw=black,fill=black,inner sep=0pt, minimum size=1pt}}
				\tikzset{edge/.style={black, line width=.4mm}}		

                \node[black node] (1) at (0,0){\color{white}$c_4$};
                \node[black node] (2) at (0,1){\color{white}$c_3$};
                \node[black node] (3) at (0,2){\color{white}$c_4$};                
                \node[black node] (4) at (1.5,1.6){\color{white}$c_5$};            
                \node[black node] (5) at (2.1,1.3){\color{white}$c_6$};             
                \node[black node] (6) at (2.6,1){\color{white}$c_7$};               
                \node[black node] (7) at (4.0,0){\color{white}$c_1$};               
                \node[black node] (8) at (2.5,0){\color{white}$c_2$};               
                \node[black node] (9) at (1.8,0){\color{white}$c_3$};               
                \node[black node] (10) at (-1.6,1.5){\color{white}$c_3$};           
                \node[black node] (11) at (-2.2,1.2){\color{white}$c_2$};           
                \node[black node] (12) at (-4.0,0){\color{white}$c_1$};                             \node[black node] (13) at (-2.6,0){\color{white}$c_7$};             
                \node[black node] (14) at (-2,0){\color{white}$c_6$};
                \node[black node] (15) at (-1.4,0){\color{white}$c_5$};

                \node[invisible node] (16) at (4.5, -0.2){};
                \node[invisible node] (17) at (-4.5, -0.2){};

                \draw[edge] (1)--(2);
                \draw[edge] (2)--(3);
                \draw[edge] (3)--(4);
                \draw[edge] (4)--(5);
                \draw[edge] (5)--(6);
                \draw[edge] (6)--(7);
                \draw[edge] (7)--(8);
                \draw[edge] (8)--(9);
                \draw[edge] (9)--(1);
                \draw[edge] (3)--(10);
                \draw[edge] (10)--(11);
                \draw[edge] (11)--(12); 
                \draw[edge] (12)--(13);
                \draw[edge] (13)--(14);
                \draw[edge] (14)--(15); 
                \draw[edge] (15)--(1);
                \draw[edge] (10)--(11);
                \draw[edge] (11)--(12);   
                \draw[edge] (7)--(16);
                \draw[edge] (12)--(17);
                \node[] at (0,2.5) {$v$};
                \node[] at (0,-0.5) {$a$};
                \node[] at (-4.5,0.2) {$b$};   
                \node[] at (4.5,0.2) {$c$};
                \node[] at (0.5,0.7) {$S_a$};
                \node[] at (-2.8,1.5) {$S_b$};
                \node[] at (3,1.5) {$S_c$};
			\end{tikzpicture}
             \begin{tikzpicture}[scale=0.75]	
			\tikzset{black node/.style={shape=circle,draw=black,fill=black,inner sep=0pt, minimum size=11pt}}
            \tikzset{white node/.style={shape=circle,draw=black,fill=white,inner sep=0pt, minimum size=11pt}}
			\tikzset{invisible node/.style={shape=circle,draw=black,fill=black,inner sep=0pt, minimum size=1pt}}
				\tikzset{edge/.style={black, line width=.4mm}}		

                \node[black node] (1) at (0,0){\color{white}$c_6$};
                \node[black node] (2) at (0,1){\color{white}$c_5$};
                \node[black node] (3) at (0,2){\color{white}$c_4$};                
                \node[black node] (4) at (1.5,1.6){\color{white}$c_3$};            
                \node[black node] (5) at (2.1,1.3){\color{white}$c_2$};             
                \node[black node] (6) at (2.6,1){\color{white}$c_1$};               
                \node[black node] (7) at (4.0,0){\color{white}$c_2$};               
                \node[black node] (8) at (2.5,0){\color{white}$c_1$};               
                \node[white node] (9) at (1.8,0){\color{black}$c_7$};               
                \node[black node] (10) at (-1.6,1.5){\color{white}$c_3$};           
                \node[black node] (11) at (-2.2,1.2){\color{white}$c_2$};           
                \node[black node] (12) at (-4.0,0){\color{white}$c_1$};                             \node[black node] (13) at (-2.6,0){\color{white}$c_7$};             
                \node[black node] (14) at (-2,0){\color{white}$c_1$};
                \node[black node] (15) at (-1.4,0){\color{white}$c_7$};

                \node[invisible node] (16) at (4.5, -0.2){};
                \node[invisible node] (17) at (-4.5, -0.2){};

                \draw[edge] (1)--(2);
                \draw[edge] (2)--(3);
                \draw[edge] (3)--(4);
                \draw[edge] (4)--(5);
                \draw[edge] (5)--(6);
                \draw[edge] (6)--(7);
                \draw[edge] (7)--(8);
                \draw[edge] (8)--(9);
                \draw[edge] (9)--(1);
                \draw[edge] (3)--(10);
                \draw[edge] (10)--(11);
                \draw[edge] (11)--(12); 
                \draw[edge] (12)--(13);
                \draw[edge] (13)--(14);
                \draw[edge] (14)--(15); 
                \draw[edge] (15)--(1);
                \draw[edge] (10)--(11);
                \draw[edge] (11)--(12);   
                \draw[edge] (7)--(16);
                \draw[edge] (12)--(17);
                \node[] at (0,2.5) {$v$};
                \node[] at (0,-0.5) {$a$};
                \node[] at (-4.5,0.2) {$b$};   
                \node[] at (4.5,0.2) {$c$};
                \node[] at (0.5,0.7) {$S_a$};
                \node[] at (-2.8,1.5) {$S_b$};
                \node[] at (3,1.5) {$S_c$};
			\end{tikzpicture}
            \vskip 6mm
            \begin{tikzpicture}[scale=0.75]	
			\tikzset{black node/.style={shape=circle,draw=black,fill=black,inner sep=0pt, minimum size=11pt}}
            \tikzset{white node/.style={shape=circle,draw=black,fill=white,inner sep=0pt, minimum size=11pt}}
			\tikzset{invisible node/.style={shape=circle,draw=black,fill=black,inner sep=0pt, minimum size=1pt}}
				\tikzset{edge/.style={black, line width=.4mm}}		

                \node[black node] (1) at (0,0){\color{white}$c_5$};
                \node[black node] (2) at (0,1){\color{white}$c_6$};
                \node[black node] (3) at (0,2){\color{white}$c_7$};                
                \node[black node] (4) at (1.5,1.6){\color{white}$c_1$};            
                \node[black node] (5) at (2.1,1.3){\color{white}$c_2$};             
                \node[black node] (6) at (2.6,1){\color{white}$c_3$};               
                \node[black node] (7) at (4.0,0){\color{white}$c_4$};               
                \node[black node] (8) at (2.5,0){\color{white}$c_5$};               
                \node[black node] (9) at (1.8,0){\color{white}$c_4$};               
                \node[black node] (10) at (-1.6,1.5){\color{white}$c_1$};           
                \node[black node] (11) at (-2.2,1.2){\color{white}$c_7$};           
                \node[black node] (12) at (-4.0,0){\color{white}$c_1$};                             \node[black node] (13) at (-2.6,0){\color{white}$c_2$};             
                \node[black node] (14) at (-2,0){\color{white}$c_3$};
                \node[black node] (15) at (-1.4,0){\color{white}$c_4$};

                \node[invisible node] (16) at (4.5, -0.2){};
                \node[invisible node] (17) at (-4.5, -0.2){};

                \draw[edge] (1)--(2);
                \draw[edge] (2)--(3);
                \draw[edge] (3)--(4);
                \draw[edge] (4)--(5);
                \draw[edge] (5)--(6);
                \draw[edge] (6)--(7);
                \draw[edge] (7)--(8);
                \draw[edge] (8)--(9);
                \draw[edge] (9)--(1);
                \draw[edge] (3)--(10);
                \draw[edge] (10)--(11);
                \draw[edge] (11)--(12); 
                \draw[edge] (12)--(13);
                \draw[edge] (13)--(14);
                \draw[edge] (14)--(15); 
                \draw[edge] (15)--(1);
                \draw[edge] (10)--(11);
                \draw[edge] (11)--(12);   
                \draw[edge] (7)--(16);
                \draw[edge] (12)--(17);
                \node[] at (0,2.5) {$v$};
                \node[] at (0,-0.5) {$a$};
                \node[] at (-4.5,0.2) {$b$};   
                \node[] at (4.5,0.2) {$c$};
                \node[] at (0.5,0.7) {$S_a$};
                \node[] at (-2.8,1.5) {$S_b$};
                \node[] at (3,1.5) {$S_c$};
			\end{tikzpicture}
    \end{center}
	\caption{Figure for Claim \ref{(3,2,1)}. Extensions of $\phi$ to $G$. }
	\label{fig:321s}
\end{figure}

We may therefore assume that $k=1$. Since $a$ is not contained in a cell, $a$ is not incident with a 4-string by Lemma \ref{4str}. Since $a$ is incident with two 1-strings and has weight at least 5, it follows that $a$ is a vertex of type (3,1,1). Note since $a$ has weight 5, it follows from the discharging rules that $a$ shares each of its short strings with a vertex that is poor immediately after Step 4. Otherwise $a$ sends charge to $v$ in Step 5. Let $d \neq a$ be an endpoint of $S_d$. Note $d$ is adjacent to $c$ which has degree at least 3 since it is the endpoint of a string. Thus $d$ is adjacent to a 0-string and a 1-string. Since $ch_4(d) < 0$, it follows that $d$ has degree 3 and weight at least 6. But then $d$ is adjacent to a $k$-string with $k \geq 5$, contradicting Lemma \ref{max-str}.

We may therefore assume $S_{ab}$ is not a 3-string. But then $S_{ab}$ contributes at most 2 to $\textrm{wt}(a) + \textrm{wt}(b)$. Let $S_1$ be the third string incident with $a$, with $S_1 \not \subset S_{ab}$ and $S_1 \neq S_a$. Similarly, let $S_2$ be the third string incident with $b$, with $S_2 \not \subset S_{ab}$ and $S_2 \neq S_b$. Let $m_a$ and $m_b$ be integers chosen such that $S_1$ is an $m_a$-string, and $S_2$ is an $m_b$-string. Since $a$ is incident with a 1-string $S_a$ and $b$ is incident with a 2-string $S_b$, we have $\textrm{wt}(a) + \textrm{wt}(b) \leq 2 + 2 + 1 + m_a + m_b$. Since each of $a$ and $b$ has weight at least 5, it follows that $m_a + m_b \geq 5$. Hence at least one of $m_a$ and $m_b$ is at least three. Suppose first $m_b \geq 3$. Note $m_b \leq 3$, since otherwise by Lemma \ref{4str} $b$ is contained in a cell, contrary to assumption.

First suppose $b$ is a vertex of type $(3,2,0)$. Then since $b$ has weight 5, it shares each of its short strings with a poor vertex as otherwise $b$ sends charge to $v$ in Step 5. Thus $b$ shares its 0-string with a vertex $w$ of degree 3 and weight at least six. Note $w \in S_{ab}$, and $S_{ab}$ contributes at most 2 to the weight of $w$ by assumption. But then since $w$ has weight at least 6, it is the endpoint of a 4-string and so is contained in a cell. This is a contradiction, as vertices contained in cells are not poor after Step 1. 

We may therefore assume that $b$ is either of type $(3,2,1)$ or of type $(3,2,2)$.  But then $S_{ab}$ contributes at most 1 to the weight of $a$. Since $a$ has weight at least 5 and is not contained in a cell, we have that $m_a = 3$, and $S_{ab}$ contributes 1 to the weight of each of $a$ and $b$. Note $\textrm{wt}(a) = 5$. By assumption, $a$ shares both its short strings with vertices that are poor after Step 4. Let $w' \in S_{ab}$ be the vertex of degree at least 3 that shares a 1-string with each of $a$ and $b$. Since $w'$ is poor after Step 4, it has degree 3 and weight at least 6. But then $w'$ is incident with a 4-string, and so by Lemma \ref{4str} it is contained in a cell. This is a contradiction, as vertices contained in cells are not poor after Step 1. 

Thus we may assume $m_b \leq 2$, and so $m_a \geq 3$. Since $a$ is not contained in a cell, $a$ is not incident with a 4-string by Lemma \ref{4str}. Thus $m_a = 3$. Note since $\textrm{wt}(a) \geq 5$, we have that $S_{ab}$ contributes at least 1 to the weight of $a$. Thus $a$ is either of type $(3,1,1)$ or (3,2,1).

Suppose first $a$ is of type (3,1,1).  By the discharging rules, since $ch_5(v) < 0$ it follows that $a$ shares a 1-string with a vertex $w'' \in V(S_{ab})$ such that $ch_4(w'') < 0$. But then $\textrm{wt}(w'') \geq 6$ and $w''$ has degree 3. Since $S_{ab}$ contributes at most 2 to the weight of $w''$, it follows that $w''$ is incident with an $r$-string with $r \geq 4$.  But this is a contradiction, as by Lemma \ref{4str} vertices of degree 3 incident with 4-strings are contained in cells. 

Thus we may assume $a$ is of type (3,2,1). But then $S_{ab}$ contributes 0 to the weight of $b$. Since $b$ is not contained in a cell, $b$ is not incident with a 4-string by Lemma \ref{4str}. Thus it follows that since $\textrm{wt}(b) \geq 5$, $b$ is of type $(3,2,0)$. Note since $ch_5(v) < 0$ and $\textrm{wt}(b) = 5$, it follows from Rule 5 that $b$ shares its 0-string with a vertex $w^*$ of degree 3 and weight at least 6, such that $ch_4(w^*)<0$.  Thus $w^*$ is not contained in a cell. But since $S_{ab}$ contributes at most 2 to the weight of $w^*$ and $w^*$ has weight at least 6, it follows that $w^*$ is incident with a 4-string. By Lemma \ref{4str}, $w^*$ is contained in a cell \textemdash a contradiction.
\end{proof}

The only remaining possibility is then that $v$ is a vertex of type (3,3,0). Let $u$ be the vertex that shares a 0-string with $v$. Since $ch_5(v) < 0$, we have that $u$ has degree 3 and weight at least 5. Note $u$ is not incident with a $4$-string as otherwise by Lemma \ref{4str} $v$ is contained in a cell, a contradiction. Since $u$ is incident with a 0-string, it follows that $u$ is either of type (3,3,0) or (3,2,0). Suppose first $u$ is of type (3,2,0). Since $ch_5(v) < 0$, and $\textrm{wt}(u)=5$, it follows that $u$ shares its incident 2-string with another vertex $w$ with $ch_4(w) < 0$. Otherwise, $u$ shares only one short string with a vertex that is poor after Step 4: thus $u$ sends $-ch_4(v)$ to $v$ in Step 5, contradicting that $ch_5(v) < 0$.  Since $w$ is not contained in a cell, by Lemma \ref{4str} $w$ is not incident with a 4-string. It follows that $w$ is not of type $(4,2,2)$, $(4,2,1),$ or $(4,2,0)$.  Furthermore, by Lemma \ref{3,2,2s}, $w$ is not a vertex of type $(3,2,2)$. Since $w$ is the endpoint of a $2$-string, it is thus of type (3,2,1) or (2,2,2). But by Claim \ref{(3,2,1)}, if $w$ is of type (3,2,1) then $ch_5(w) \geq 0$. Similarly, by Claim \ref{(2,2,2)} $w$ is not of type (2,2,2). 

Therefore we may assume $u$ is of type $(3,3,0)$. But then by Lemma  \ref{(3,3,0)}, $u$ is contained in a cell $C$. Since $v$ receives charge from $C$ in Step 2, $ch_5(v) \geq 0$ \textemdash a contradiction.
\end{proof}

We now prove our main theorem.

\begin{proof}[Proof of Theorem \ref{main}]
Suppose not. Let $G$ be a minimum counterexample. At the end of the discharging process described above, all structures in $G$ have non-negative charge. But as shown in Equation (\ref{threshhold}), the charge carried by the graph is at most $-6$. Since the total charge did not change throughout the discharging process, this is a contradiction.
\end{proof}


\begin{section}{Open Questions}

A natural question to wonder is whether or not the density bound obtained in Theorem \ref{main} is best possible. We suspect not. Kostochka and Yancey \cite{kostochka2014ore} showed that if $G$ is $k$-critical and $k \geq 4$, then $e(G) \geq (\frac{k}{2}- \frac{1}{k-1})v(G) - \frac{k(k-3)}{2(k-1)}$.  Later, they showed this is tight for graphs\footnote{They showed further that this bound is tight \emph{only} for the graphs obtained via Ore's construction.} obtained via a construction given by Ore in \cite{ore2011four}. A $k$-critical graph given by Ore's construction is called a $k$-\emph{Ore graph}.

Given a $(2t+2)$-critical graph, there is a seemingly natural way to obtain a $C_{2t+1}$-critical graph by edge subdivisions. Indeed, we have the following:

\begin{restatable}{prop}{orecons}\label{orecons}
If $G$ is a $(2t+2)$-critical graph, then the graph $G'$ obtained from $G$ by subdividing every edge $(2t-2)$ times is $C_{2t+1}$-critical. 
\end{restatable} 

The proof of this proposition is omitted, but the key observation is the following: for each edge $uv \in E(G)$, let $P_{uv} \subset G'$ be the $(u,v)$-path obtained by subdividing $uv$ $(2t-2)$ times. If $\phi$ is a mapping from $u$ to a vertex $c$ in $C_{2t+1}$, then there exists an extension of $\phi : P_{uv} \rightarrow C_{2t+1}$ with $\phi(v) = c'$ for precisely the set $\{c': c' \in V(C)-c\}$. In this way, $P_{uv}$ restricts colourings of its endpoints in the same manner as an edge does in ordinary vertex colouring\footnote{This idea is formalized in Lemma \ref{vartheta}.}.

Since the edge-density obtained by Kostochka and Yancey for $k$-critical graphs is tight for $k$-Ore graphs, it seems reasonable that the corresponding density obtained from subdividing a $(2t+2)$-Ore graph could be best possible for $C_{2t+1}$-critical graphs. This idea motivates the following.

\begin{prop}\label{notjustdensity}
Let $t \geq 1$ be an integer, and let $G$ be a $(2t+2)$-Ore graph. Let $G'$ be the graph obtained from $G$ by subdividing each edge in $E(G)$ $(2t-2)$ times. Then $e(G') = \frac{t(2t+3)v(G')-(t+1)(2t-1)}{2t^2+2t-1}$.
\end{prop}

We therefore find it reasonable by setting $t=3$ to conjecture that if $G$ is a $C_7$-critical graph, then $e(G) \geq \tfrac{27v(G)-20}{23}$. More generally, we ask the following question.
\begin{question}\label{conj1}
Let $t \geq 3$. Does every $C_{2t+1}$-critical graph $G$ satisfy $e(G) \geq \frac{t(2t+3)v(G)-(t+1)(2t-1)}{2t^2+2t-1}$?
\end{question}

We note that the family of graphs described in Proposition \ref{notjustdensity} show that it is impossible to prove Conjecture \ref{jaeger} using only a density bound. When $t=3$ for example, the graphs described in Proposition \ref{notjustdensity} have an asymptotic density of $\frac{27}{23}$. However, using Euler's formula for planar graphs, we have that if $G$ is a planar graph of girth at least $g$, then $e(G) \leq \frac{g}{g-2}(v(G)-2)$ \textemdash or, asymptotically, that $\frac{e(G)}{v(G)} \leq \frac{g}{g-2}$. In order to obtain a density argument that implies a relaxation of Conjecture \ref{jaeger}, it follows that the girth bound $g$ chosen in the relaxation will satisfy $\frac{g}{g-2} \leq \frac{27}{23}$ \textemdash or in other words, that $g \geq 14$. A proof of Conjecture \ref{jaeger} will thus not be a purely density-based argument: it will require additional tools (for instance tools exploiting planarity).

More generally, we note that a negative answer to Question \ref{conj1} together with Euler's formula for planar graphs implies that if $G$ is a planar graph with girth at least $4t+2$, then $G$ admits a homomorphism to $C_{2t+1}$. The girth bound of $4t+2$ is of particular interest as no counterexamples to the primal version of the conjecture with edge-connectivity $4t+2$ have been found. Indeed, all counterexamples found in \cite{han2018counterexamples} are at most $(4t+1)$-edge connected.

\end{section}

\bibliographystyle{plain}
\bibliography{biblio}
\end{document}